\documentclass[11pt]{article}
\usepackage{amsmath, amssymb, amsthm}

\usepackage{hyperref}

\usepackage{graphicx} % Required for inserting images
\usepackage{amsmath, color, comment, mathtools, cleveref, enumitem}
\usepackage{xcolor}

\usepackage[margin=1.2in]{geometry}

\usepackage{float}
\usepackage[maxbibnames=99]{biblatex}
\bibliography{biblio}
\usepackage{bm} 

\newtheorem{thm}{Theorem}[section]
\newtheorem{prop}[thm]{Proposition}
\newtheorem{lemma}[thm]{Lemma}
\newtheorem{cor}[thm]{Corollary}
\newtheorem{mydef}[thm]{Definition}

\newtheorem{remark}[thm]{Remark}

\numberwithin{equation}{section} %Equation numbering

\DeclareMathOperator{\supp}{supp}

\DeclareMathOperator{\sgn}{sgn}

\newcommand{\p}{{\partial}}

\renewcommand{\d}{\delta}

\newcommand{\R}{{\mathbb{R}}}

\newcommand{\T}{{\mathbb{T}}}
\newcommand{\g}{{\mathsf{g}}}

\newcommand{\eps}{\varepsilon}

\newcommand{\D}{\Delta}

\newcommand{\nn}{\nonumber}

\newcommand{\ep}{\varepsilon}
\newcommand{\al}{\alpha}
\newcommand{\be}{\beta}

\newcommand{\la}{\lambda}

\newcommand{\om}{\omega}

\let\div\relax
\DeclareMathOperator{\div}{\mathsf{div}}

\def\XXint#1#2#3{{\setbox0=\hbox{$#1{#2#3}{\int}$ }
\vcenter{\hbox{$#2#3$ }}\kern-.6\wd0}}

\def \hal{\frac{1}{2}}
\def\({\left(}
\def\){\right)}
\def \ep{\varepsilon}
\def\nab{\nabla}
\def\indic{\mathbf{1}}

\usepackage{fancyhdr}

\title{On a repulsion model with Coulomb interaction and nonlinear mobility}

\author{Antonin Chodron de Courcel\thanks{Laboratoire Alexandre Grothendieck, Institut des Hautes Études Scientifiques, Université
Paris-Saclay, CNRS, 35 Routes de Chartres, 91440 Bures-sur-Yvette, France; Email: decourcel@ihes.fr} 
\and Charles Elbar\thanks{Université Claude Bernard Lyon 1, ICJ UMR5208, CNRS, Ecole Centrale de Lyon, INSA Lyon, Université Jean Monnet, 69622
Villeurbanne, France. Email: elbar@math.univ-lyon1.fr}
}

\date{}

\begin{document}

\maketitle

\begin{abstract}
We study a scalar conservation law on the torus in which the flux $\mathbf{j}$ is composed of a Coulomb interaction and a nonlinear mobility: $\mathbf{j} = -u^m\nab\g\ast u$. We prove existence of entropy solutions and a weak–strong uniqueness principle. We also prove several properties shared among entropy solutions, in particular a lower barrier in the fast diffusion regime $m<1$. In the porous media regime $m\ge 1$, we study the decreasing rearrangement of solutions, which allows to prove an instantaneous growth of the support and a waiting time phenomenon. We also show exponential convergence of the solutions towards the spatial average in several topologies. 

\end{abstract}
\vskip .7cm

\noindent{\makebox[1in]\hrulefill}\newline

\tableofcontents

\section{Introduction}

\subsection{Context}

In many physical and biological systems, individuals tend to move away from regions that are already crowded. For instance, charged particles in plasma repel each other, and migrating cells may alter their direction and speed after colliding: a behavior known as contact inhibition of locomotion. The evident mechanism in these examples is a repulsive field generated by the population itself, which drives individuals away from high-density regions. In addition, the ability to move may vary with local concentration, leading to a nonlinear mobility. The model we study incorporates these effects in their simplest possible form. More precisely, we consider the system posed on the $d-$dimensional torus $\T^d=\R^d\setminus\mathbb{Z}^d$:

\begin{equation}\label{eq:PDE}
\begin{cases}
    \p_t u -\div (u^m \nab\g\ast u) = 0, \quad (t,x)\in (0,\infty)\times \T^d,\\
    u\vert_{t=0} = u_0.
\end{cases}
\end{equation}

Here $u\equiv u(t,x)$ is the density of individuals, $u_{0}$ is a nonnegative initial condition, and $\g$ is the Green's kernel on the torus (assumed to be of size 1), so that the repulsive forces are assumed of Coulomb type: 

\begin{equation}\label{eq:green_kernel}
-\Delta \g = \delta_{0}-1.  
\end{equation}

Depending on the exponent $m>0$, the model describes how the density influences motion:

\underline{$m>1$}: The motion increases with the density. As a result, populations do not spread instantaneously. Instead, a front is created and it propagates with a finite speed. In some cases, this front may even remain stationary for some time, before beginning to move: a waiting time phenomenon that we describe later. This behavior captures the realistic idea that expansion can happen only once there is sufficient pressure built up behind the front. 

\underline{$m=1$}: The simplest situation, where the mobility is constant.

\underline{$0<m<1$}: The motion decreases with the density. For instance, when the density is low, individuals move more freely. As a result, the population spreads rapidly, and the density becomes instantaneously positive in the whole space.

As soon as $m\ne 1$, shocks may appear in finite time. Indeed, the equation is then a conservation law with effective velocity given by $-u^{m-1}\nab\g\ast u$. Consequently, if the solution is steep near some point, the velocity may differ between the top and bottom of this steepness by an amount of order 1. On the other hand, it is known from \cite{MR1739596} that, when $m=1$, the equation propagates any Sobolev norms.

\subsection{Related works}

In the case $m=1$, Equation~\eqref{eq:PDE} was introduced to model superconductivity or superfluidity~\cite{chapman1996mean, weinan1994dynamics, MR1739596}. From the mathematical point of view, when $m\le 1$, Equation~\eqref{eq:PDE} is a gradient flow of the energy 

\begin{equation}\label{eq:energy_gf}
\mathcal{E}[u]= \frac{1}{2}\int_{\T^d}  \g \ast u \, u\,     dx,
\end{equation}

with respect to the a modified Wasserstein metric adapted to the nonlinear mobility~\cite{MR2448650, CarrilloLisiniSavareSlepcev2010}. When $m=1$ we recover the usual Wasserstein-2 distance.  In particular, the system minimizes its energy in the best way possible with respect to this metric. When $m> 1$ these distances are not well defined as the mobility is no longer concave. 

The equation can also be seen as a repulsive version of the Keller--Segel system, a general form of which being  
\begin{equation}\label{eq:KS_general}
\partial_t u - \sigma \Delta u + \div(u\,\chi(u,v)\nabla v)=0, 
\qquad -\Delta v = u-\bar{u},
\end{equation}
where $u$ is the density of individuals, $v$ the concentration of a chemosignal, $\sigma \geq 0$ a diffusion constant, and $\chi=\chi(u,v)$ the chemical sensitivity function describing how individuals respond to the chemical gradient. This system has been derived in mathematical biology by Keller and Segel in~\cite{keller1971model,keller_initiation_1970} as a canonical model for chemotaxis and its many variants have been studied extensively both in the attractive regime (chemoattraction) and in the repulsive regime (chemorepulsion). Our system fits into this general class with the choices  
$$
\sigma =0, 
\qquad 
\chi(u,v)=-u^{\,m-1},
$$

The sign in front of the sensitivity shows that the model describes chemorepulsion, rather than chemoattraction: particles move down the chemical gradient. Moreover, the specific dependence $\chi(u,v)=-u^{m-1}$ is a choice of a nonlinear chemical sensitivity. A similar model has been studied in~\cite{MR3723527,MR3124759} with a diffusion coefficient $\sigma >0$. 

When $m=1$, the system can be rigorously derived as a mean-field limit from the system
$$
\frac{dX_i}{dt} = - \frac{1}{N}\sum_{j\ne i}\nabla \g (X_i - X_j),
$$
where $X_i$ is the position of the particle $i$ and $N$ is the number of particles. It has been obtained in \cite{serfaty2020mean} that the convegence of the empirical measure $u_N(t)=N^{-1}\sum_{i=1}^N \delta_{X_i(t)}$ to the unique sufficiently regular solution of the limiting PDE \eqref{eq:PDE} with $m=1$ is propagated at all positive times. For $m\ne 1$, it is not clear whether equation \eqref{eq:PDE} can be seen as steeming from a particle system. Nevertheless, this does not prevent from using particle methods in order to simulate the equation, such as the blob method from~\cite{MR3471104, MR3913840}, the Cloud-in-Cell method (CIC) or front tracking methods.

The equation \eqref{eq:PDE} has been previously studied in \cite{CGV22FastRegularisation} when $m\in (0,1)$, and in \cite{CGV2022vortex} when $m>1$. In these works, the authors develop a rather complete theory of \emph{radial} solutions on the whole space $\R^d$. Their strategy is to use the link between the equation and its Hamilton-Jacobi formulation. More precisely, assuming $d=1$, they consider the equation
\begin{equation}\label{eq:primitiveCarillo}
    \p_t k + (\p_s k)_+^m k = 0, \quad (t,x) \in (0,\infty)^2,
\end{equation}
which is satisfied by $\displaystyle k(t,s) := \int_0^s u(t,\sigma)d\sigma$. For general dimensions $d\ge 2$, their theory naturally extends if one considers radial solutions. The advantage of considering \eqref{eq:primitiveCarillo} is that this equation satisfies a comparison principle, so that we can use the framework of viscosity solutions and develop a theory for \eqref{eq:PDE} based on that for \eqref{eq:primitiveCarillo}. In particular, the uniqueness result is stated in terms of \eqref{eq:primitiveCarillo}, not \eqref{eq:PDE}. Moreover, one can compute the characteristics of \eqref{eq:primitiveCarillo} in order to have analytical expressions of some specific solutions. In this article, we develop a theory for non-radial solutions to \eqref{eq:PDE}. 

We finally note several features of Equation~\eqref{eq:PDE}. First, we expect shocks to appear in finite time as soon as $m\ne 1$. For general dimensions $d\ge 2$, it is not clear which kind of regularity could be propagated by the equation, although some specific solutions give a hint (see the Appendix). Second, this equation does not satisfy a comparison principle in general. Lastly, in some particular cases (for instance if we do not assume a lower bound when $m<1$), it is not clear whether one can always provide a compactness criterion allowing to build weak solutions.  

Another realm of models consider interactions that are more singular than \eqref{eq:green_kernel}. That is, $\g$ is now given by $\g(x) \sim  |x|^{-s}$, for $(d-2)_+< s< d$. In this situation, the model becomes hypoelliptic and it is unclear whether this dominates the formation of shocks. We refer to \cite{stanTesoVazquez2019existence, STAN2016infinitespeed, nguyen2018porous, ACDC25}. In particular, some of our results use the strategy of \cite{ACDC25}, in particular to derive lower bounds in the fast diffusion case $m<1$ (see Theorem \ref{thm:entropysolnproperties}).

\subsection{Main results}

We summarize here the main contributions of the paper. Our first result is the existence of entropy solutions to~\eqref{eq:PDE}. The definition is as follows.

\begin{mydef}\label{def:weaksolution}
A function $u$ is an entropy solution of~\eqref{eq:PDE} if
\begin{enumerate}
    \item $u\ge 0$ and $u\in L^{\infty}(\R_{+}\times\T^d)$,
    \item for every $T>0$ and $\varphi \in C^{1}_c([0,T)\times\T^d)$,
   \begin{equation*}
        \int_{0}^{T}\int_{\T^d}
        \left[-u\,\partial_t\varphi
              +u^m\,\nabla\g\ast u\cdot\nabla\varphi \right] dx\, dt
          =\int_{\T^d}u_0(x)\varphi(0,x)\,dx 
      \end{equation*}
      \item for all $\eta\in C^2$ convex and $q'(\xi) := m\eta'(\xi)\xi^{m-1}$, we have in the sense of distributions
      \begin{equation}\label{eq:entropycondition}
          \p_t \eta(u)\le \div(q(u)\nab\g\ast u) +(q(u)-\eta'(u) u^m)(u-\bar u).
      \end{equation}
\end{enumerate}
\end{mydef}

The first result concerns the existence of entropy solutions for the model \eqref{eq:PDE}.

\begin{thm}[Existence of entropy solutions]\label{thm:existence}
Suppose $u_{0}\in L^{\infty}(\T^d)$ is nonnegative, with the additional condition $u_{0}>0$ when $0<m<1$. Then, there exists at least one entropy solution $u$ of~\eqref{eq:PDE} in the sense of Definition~\ref{def:weaksolution}. This solution moreover satisfies
\begin{enumerate}
    \item (conservation of mass) for a.e. $t>0$,
    \begin{equation*}
        \int_{\T^d} u(t) dx = \int_{\T^d} u_0dx =: \bar u_0,
    \end{equation*}
    \item (decreasing of $L^p$ norms) for all $1\le p\le \infty$ and a.e. $t>0$, 
    \begin{equation*}
        \|u(t)\|_{L^p}\le \|u_0\|_{L^p}, 
    \end{equation*}
    \item (energy dissipation) for a.e. $t>0$,
    \begin{equation*}
        \hal\int_{\T^d}\g\ast u(t) u(t) dx + \int_{0}^t\int_{\T^d} |\nab\g\ast u|^2 u^m dx\, d\tau \le \hal \int_{\T^d}\g\ast u_0 u_0dx.
    \end{equation*}
\end{enumerate}
\end{thm}

Uniqueness is known only in particular cases: in one spatial dimension, or for radial solutions, as shown in~\cite{CGV2022vortex, CGV22FastRegularisation}. In the latter case, one can exploit the classical connection between one-dimensional scalar conservation laws and Hamilton–Jacobi equations, for which the theory of viscosity solutions applies. The question of uniqueness remains open in the general setting, even under entropy assumptions. Instead, we establish a weak–strong uniqueness principle: whenever a sufficiently regular (strong) solution to~\eqref{eq:PDE} exists, every entropy solution with the same initial data coincides with it. In other words, the strong solution is unique within the broader class of entropy solutions.

\begin{prop}[Weak strong uniqueness]\label{thm:weak_strong_uniqueness}
Let $u,v\in L^\infty(\R_+, L^\infty(\T^d))$ be two entropy solutions of~\eqref{eq:PDE} such that, for all $T>0$,
\begin{equation}\label{eq:est_weak_strong}
    \displaystyle \int_0^T \| \nab v^{m-1}(t)\|_{L^{d,1}}dt <+\infty.
\end{equation}
Then, there exists $C>0$ depending on $d,s$ such that
\begin{equation*}
    \| (u-v)(t)\|_{L^1} \\
    \le \| u_0-v_0\|_{L^1} \exp\bigg( C\int_0^t \|v(\tau) \|_{L^\infty} 
    \| \nab v^{m-1}(\tau)\|_{L^{d,1}}  d\tau \bigg).
\end{equation*}
\end{prop}

\begin{remark}
When $d\ge 2$, the space $W^{1,(d,1)}(\T^d)$ is continuously embedded in $C^0(\T^d)$. However, as we expect shocks, we don't expect the solutions to remain continuous. 
\end{remark}

Provided the initial condition is smooth enough (more precisely, $u_{0}\in W^{1,\infty}(\T^d)$), the Cauchy problem \eqref{eq:PDE} is locally well-posed.

\begin{prop}\label{prop:strong_solution}Let $d\ge 1$ and $m,c>0$. Assume either $u_{0}\ge c$ or $m\ge2$. Let $u$ be an entropy solution to \eqref{eq:PDE} constructed in Theorem~\ref{thm:existence} with initial datum $u_{0}\in W^{1,\infty}(\T^d)$. Then there exists a time $T_* >0$ depending on the initial condition 
such that, 
$$
\nabla u \in L^{\infty}(0,T_*\, ; L^{\infty}(\T^d)).
$$
In particular, this solution is the unique entropy solution on $[0,T_*)$.
\end{prop}

Although uniqueness is an open problem in general, we prove that any entropy solution satisfies the following estimates, obtained in Section 3.
\begin{thm}\label{thm:entropysolnproperties}
    Any entropy solution $u$ of \eqref{eq:PDE} with initial condition $u_0 \ge 0$ satisfies the following estimates:
    \begin{enumerate}
        \item (instantaneous regularization) 
        \begin{equation*}
            u(t,x) \le \bar u_0 + (mt)^{-1/m}, \quad \text{for a.e. } t>0, x\in \T^d,
        \end{equation*}
        \item (decreasing of $L^p$ norms)
        \begin{equation*}
            t\mapsto \| u(t)\|_{L^p} \quad \text{is nonincreasing,}
        \end{equation*}
        \item (lower barrier) if $m<1$, then
        \begin{equation*}
            u(t,x) \ge C\min ( (\bar u_0 t)^\frac{1}{1-m}, \bar u_0 (1- e^{-C t})).
        \end{equation*}
    \end{enumerate}
\end{thm}
\begin{remark}
    The lower bound obtained for any entropy solution in the case $m<1$ is the main point of this theorem. Indeed, there exists weak (non-entropic) solutions to \eqref{eq:PDE} that are not positive in the case $m<1$ (see \cite{CGV22FastRegularisation}).
\end{remark}
\begin{remark}
    The scaling law $t^\frac{1}{1-m}$ obtained in the lower bound is that of self-similar solutions (see \cite{CGV22FastRegularisation}).  
\end{remark}

A sharp difference between the cases $m<1$ and $m\ge 1$ concerns the evolution of the free boundary of the support. Indeed, as shown in Theorem \ref{thm:entropysolnproperties}, any entropy solution fills the whole domain instantaneously when $m<1$. By contrast, there is a moving front when $m\ge 1$. In the radial case \cite{CGV2022vortex}, the authors characterize this phenomenon as follows: depending on the regularity of the initial data near its boundary, there may be a waiting time during which the support does not expand, or the free boundary may move instantaneously. This characterization is made possible through the use of an integrated (Hamilton-Jacobi-type) equation for the density, through the quantity
\begin{equation*}
    k(t,s ) := \int_0^s u(t,\sigma)d\sigma, \quad u(t,x) \equiv u(t,|x|).
\end{equation*}
As already mentioned, such a tool is not available in our nonradial setting. However, we can recover part of their result by using decreasing rearrangements of solutions. 

The idea is as follows. Denoting by $u_*(t,s) := (u(t))_*(s)$ the decreasing rearrangement of $u$\footnote{We introduce these notations in detail in Section~4.}, we define 
\begin{equation*}
    k(t,s) := \int_0^s u_*(t,\sigma)d\sigma.
\end{equation*}
Studying $k$ is, formally, the analogue of the Hamilton–Jacobi formulation available in the radial case.

\begin{thm}[Growth of the support]\label{thm:waitingtime}
   Let $d\ge 1$, $m\ge1$. Consider $u_0\in L^\infty(\T^d)$ and a solution $u$ to \eqref{eq:PDE} starting at $u_0$ constructed via a vanishing viscosity method as in Theorem~\ref{thm:existence}. Define the quantity
   $$S(t):=|\supp u(t)| = |\supp u_{*}(t)|.$$
   Let $S_0:=S(0) <1$. Under the condition \begin{equation}\label{eq:condition_waiting}
        \limsup_{s\to S_0^{-}} \frac{1}{(S_0-s)^{\frac{m}{m-1}}}\int_{ s<\sigma<S_0} (u_0)_{*}(\sigma)d\sigma =+\infty,
    \end{equation}
    there is no waiting time, i.e. for a.e. $t>0$, $S(t)>S_0$.

    By contrast, if $u^{m-1}\in C^{0,1}([0,T_*)\times \T^d)$ for some $T_*>0$, then $S(t) = S_0$ for all $t<T_*$.
\end{thm}

\begin{remark}
The condition \eqref{eq:condition_waiting} is expected from \cite{CGV2022vortex} to be not only sufficient but necessary for the free boundary to move instantaneously. If $u_{0}$ is Lipschitz and either $m\ge 2$ or $u_{0}$ is positive, then Proposition~\ref{prop:strong_solution} yields automatically $u^{m-1}\in C^{0,1}([0,T_*)\times \T^d)$ for some $T_{*}$.
\end{remark}

Finally, we describe the asymptotic behaviour of solutions. We show exponential convergence towards the spatial average $\bar{u}_{0}$.

\begin{thm}[Asymptotic behaviour]\label{thm:asymptotic_behaviour} Let $u$ be any entropy solution of~\eqref{eq:PDE} with initial condition $u_{0}\in H^{-1}(\T^d)\cap L^{\infty}(\T^d)$. 

When $m<1$, we have for a.e. $t>0$,
\begin{align*}
    \|u(t)-\bar u_0\|_{L^1}   &\le \|u_0-\bar u_0\|_{L^1} e^{-\bar u_0^m t} \\
    \|u(t)-\bar{u}_{0}\|_{L^\infty} &\le (\|u_{0}\|_{L^\infty} - \bar{u}_{0}) e^{-\bar{u}_{0}^m t} + \hal \bar{u}_{0} e^{-2^{-m}\bar{u}_{0}^m t}, \\
    \|u(t)-\bar{u}_{0}\|_{\dot H^{-1}} &\le  \|u_{0}-\bar{u}_{0}\|_{\dot H^{-1}} e^{-C_m t}.
\end{align*}

When $m\ge 1$, assume further that $u_{0}\ge c>0$. We have for a.e. $t>0$,
    \begin{align*}
    \|u(t)-\bar u_0\|_{L^1}& \le \|u_0-\bar u_0\|_{L^1} e^{-\bar u_0^m t} \\
        \| u(t)-\bar{u}_{0}\|_{L^\infty} &\le (\|u_{0}\|_{L^\infty} - \bar{u}_{0}) e^{-\bar{u}_{0}^m t}  + (\bar{u}_{0}- c) e^{-c^mt}, \\
         \|u(t) -\bar{u}_{0}\|_{\dot H^{-1}} &\le \|u_{0} - \bar{u}_{0}\|_{\dot H^{-1}} e^{-c^mt}.
    \end{align*}
\end{thm}
\begin{remark}
    Interpolating between $L^1$ and $L^\infty$, we do not need the lower bound assumption $u_0 \ge c>0$ for the $L^p$ topologies, $1\le p <\infty$.
\end{remark}

We conclude this introduction by mentioning different lines of work, that generalize the previous model. In particular, for applications, it could be interesting to generalize the results to systems of interaction, with different populations, and also generalizing the pressure law.

\subsection{Possible extensions}

\paragraph{Systems of interactions}
We are interested in the properties of a system of interaction, where each population is influenced by a gradient of a pressure depending on the sum of the two densities.  

\begin{equation}\label{eq:system_PDE}
\begin{cases}
    \p_t u_1 -\div (u_1^m \nab\g\ast (u_1+u_2)) = 0, \quad (t,x)\in (0,\infty)\times \T^d,\\
    \p_t u_2 -\div (u_2^m \nab\g\ast (u_1+u_2)) = 0, \quad (t,x)\in (0,\infty)\times \T^d,\\
    u_{1}\vert_{t=0} = u_{1,0},\, u_{2}\vert_{t=0} = u_{2,0}.
\end{cases}
\end{equation}

It models two populations, where each agent follows the same potential created by the total density, but each population's ability to move depends on its own density through the nonlinear mobility. Mathematically, for $0<m<1$, this system is still a gradient flow of an energy, for a Wasserstein distance in a product space. It can be shown that there exists global weak solutions when $m\ge 1$ by adapting the method for a single population. However, even if we assume positive initial data, the existence of weak solutions remains open in the case $m<1$. The difficulty comes from the inability to prevent the densities from vanishing, which is a major difference compared to the single-population case. The asymptotic behaviour of each population also remains an open problem, although one expects their sum to converge to its spatial average.

\paragraph{Polytropic pressure law}

Another possible extension is the following model
\begin{equation*}
\begin{cases}
\displaystyle\partial_tu - \div (u^m \nab \g \ast u^\gamma )=0 
& \text{in } (0,\infty) \times \T^d,\\
u|_{t=0} = u_0.
\end{cases}
\end{equation*}

Here the force is a gradient of a nonlocal pressure potential, where the pressure is nonlinear with respect to the density. This system has been studied in~\cite{DaoDiaz2025} in the case $m=1$. \\

\textbf{Notations and functional settings.} We denote by $L^{p}(\T^d)$, $W^{1,p}(\T^d)$, $H^{s}(\T^d)=W^{s,2}(\T)$ the usual Lebesgue and Sobolev spaces, and by $\|\cdot\|_{L^{p}}$, $\|\cdot\|_{W^{1,p}}$ their corresponding norms. The Lorentz spaces are denoted by $L^{p,q}(\T^d)$ with $\|\cdot\|_{L^{p,q}}$ their corresponding norms. $BV(\T^d)$ is the space of bounded variation functions on the torus.  
We often write $C$ for a generic constant appearing in the different inequalities. Its value can change from one line to another, and its dependence to other constants can be specified by writing $C_{a,...}$ if it depends on the parameter $a$ and other parameters. 

\section{Existence of entropy solutions and weak-strong uniqueness}

In this section, we construct entropy solutions to \eqref{eq:PDE} via a vanishing viscosity method. We then provide a weak-strong uniqueness result, and finally show that the Cauchy problem is locally well-posed (in time), provided the initial condition is Lipschitz continuous.

\subsection{Viscous approximation}

We construct entropy solutions to \eqref{eq:PDE} via a vanishing viscosity method. The well-posedness of the viscous problem is classical. Nevertheless, passing to the limit $\varepsilon \to 0$ in the nonlinear mobility requires some strong compactness. Such a result is delicate since no Sobolev norm is proved to be propagated by the equation. We note that the situation is simpler in the more singular case where $\g(x) \sim |x|^{-s}$ as $x\to 0$, $d>s>(d-2)_+$. Indeed, the equation then becomes hypoelliptic. In our case, two possible strategies are available: one relies on the kinetic formulation of conservation laws (as in~\cite{perthame_existence_2006}), and the other on techniques developed for continuity equations (see~\cite{belgacemjabin2013}). In this work we adopt the latter approach.

Consider the following regularized problem
\begin{equation}\label{eq:viscousPDE}
\begin{cases}
    \p_t u_\ep -\div (u^m_\ep \nab\g\ast u_\ep) = \ep \D u_\ep, \quad (t,x)\in (0,\infty)\times \T^d,\\
    u\vert_{t=0} = u_{0,\ep}.
\end{cases}
\end{equation}

Here $u_{0,\ep}=u_{0}\ast\omega_{\eps}$ for some smooth mollifiers $\omega_{\eps}$. In particular, $u_{0,\eps}\to u_{0}$ in $L^{1}(\T^d)$ as $\eps\to 0$. We admit the existence of a unique classical solution to this system, but a proof on the whole space can be found in~\cite{CGV22FastRegularisation}. 

Let us first introduce the ODE with which the solutions to \eqref{eq:PDE} will be compared. The proof of the following lemma is postponed to the end of the subsection.
\begin{lemma}\label{eq:ODE}
    Let $\bar u_\ep >0$ and $0\le\be \le + \infty$ be two constants. Consider the equation
    \begin{equation*}
    \begin{cases}
         \displaystyle \frac{d}{dt} \Phi_\be(t) = \Phi_\be(t)^m(\bar u_\ep - \Phi_\be(t)), \\
         \Phi_\be(t=0) = \beta.
    \end{cases}
    \end{equation*}
    If $\bar u_\ep <\be$, the unique solution decreases and satisfies
    \begin{equation*}
        \forall t\ge0, \quad \bar u_\ep \le \Phi_{\be}(t) \le \begin{cases}
            \bar u_\ep +  (tm + (\be -\bar u_\ep)^{-m})^{-1/m}, \\
            \bar u_\ep + (\be - \bar u_\ep) e^{-\bar u_\ep^m t}.
        \end{cases}
    \end{equation*}
    If $\be < \bar u_\ep$ and $m\ge 1$, the unique solution is nondecreasing and satisfies
    \begin{equation*}
        \forall t\ge 0, \quad \bar u_\ep - (\bar u_\ep - \be) e^{-\be^m t} \le \Phi_\be(t) \le \bar u_\ep.
    \end{equation*}    
    If $\be < \bar u_\ep$ and $m<1$, there is a unique \emph{positive increasing} solution, for which we define
    \begin{equation*}
        \tau_{1/2} := \inf\{t\ge 0 : \Phi_\be(t) \ge \hal \bar u_\ep \},
    \end{equation*}
    and satisfies
    \begin{equation*}
        \forall t\ge 0, \quad \bar u_\ep \ge \Phi_\be(t) \ge \begin{cases}
            \displaystyle \bar u_\ep - \hal\bar u_\ep e^{-2^{-m}\bar u_\ep^m t}, & t\in (\tau_{1/2},\infty),\\
            \displaystyle (\be^{1-m } + \hal\bar u_\ep t)^{1/(1-m)} ,& t\in [0,\tau_{1/2}].
        \end{cases}
    \end{equation*}
\end{lemma}

The following bounds on the solution of \eqref{eq:viscousPDE} are uniform in the viscosity parameter.
\begin{prop}\label{prop:instantbounds}
    Let $d\ge 1$, $m>0$, and $u_{0,\ep} \in L^\infty(\T^d)$ such that $u_{0,\ep} \ge 0$. The unique smooth solution $u_\ep$ to \eqref{eq:viscousPDE} with initial condition $u_{0,\ep}$ satisfies, for all $t>0$ and $1\le p\le \infty$,
    \begin{enumerate}[label=\roman*.]
        \item  $\displaystyle \int_{\T^d} u_\ep(t) dx = \int_{\T^d} u_{0,\ep}dx$, 
        \item  $u_\ep(t)\ge 0$ on $\T^d$,
        \item  $t\mapsto \min u_\ep(t)$ (resp. $t\mapsto\max u_\ep(t)$) is nondecreasing (resp. nonincreasing),
        \item  $\|u_\ep(t)\|_{L^p} \le \|u_{\ep,0}\|_{L^p}$,
        \item  $\|u_\ep(t)\|_{L^\infty} \le \Phi_{\|u_{0,\ep}\|_{L^\infty}}(t)$ and $\min u_\ep(t) \ge \Phi_{\min u_{0,\ep}}(t)$, where $\Phi_{\|u_{0,\ep}\|_{L^\infty}}$ and $\Phi_{\min u_{0,\ep}}$ are defined in Lemma \ref{eq:ODE},
        \item $\displaystyle \hal \int_{\T^d} \g\ast u_\ep(t) u_\ep(t) dx  + \int_0^t\int_{\T^d} |\nab\g\ast u_\ep|^2(\tau) u_\ep^m(\tau)dxd\tau \le \hal\int_{\T^d} \g\ast u_{\ep,0} u_{\ep,0} dx$.
    \end{enumerate}
\end{prop}
\begin{remark}\label{remark:support}
    We deduce from the fifth point that $u_\ep$ has full support inside $\T^d$ for positive times when $m<1$. More precisely, we have derived a lower barrier on solutions:
    \begin{equation*}
        \forall t\ge 0, \forall x\in\T^d, \quad u_\ep(t,x) \ge \begin{cases}
            \displaystyle \bar u_\ep - \hal\bar u_\ep e^{-2^{-m}\bar  u_\ep^m t}, & t\in (\tau_{1/2},\infty),\\
            \displaystyle \left(\min (u_{\eps,0})^{1-m } + \hal\bar u_\ep t\right)^{1/(1-m)} ,& t\in [0,\tau_{1/2}].
        \end{cases}
    \end{equation*}
    Furthermore, the short-time scaling $t^{1/(1-m)}$ is that of self-similar solutions to \eqref{eq:PDE}, as given in \cite{CGV22FastRegularisation} for the Euclidean setting.
\end{remark}
\begin{remark}
    If we do not assume the initial condition $u_{0,\ep}$ to be bounded, we still obtain some kind of regularization, since 
    \begin{equation*}
        \forall t> 0, \quad \|u_\ep(t) \|_{L^\infty} \le \Phi_{+\infty}(t) \le \bar u_\ep + (tm)^{-1/m}.
    \end{equation*}
\end{remark}
We will need the following lemma, which can be found in \cite[Lemma 2.5]{ACDC25}.
\begin{lemma}\label{lem:mincurve}
    Let $T>0$ and $\mu\in C^1( [0,T]\times \T^d)$. Consider $\bar x: [0,T]\to \T^d$ such that, for all $t\in [0, T]$, $\bar x(t)$ is a minimum (resp. maximum) point of $\mu(t,\cdot)$. Assume moreover that $t\mapsto \min \mu(t,\cdot)$ is nondecreasing (resp. $t\mapsto \max \mu(t,\cdot)$ is nonincreasing). Then, $t\mapsto \mu(t,\bar x(t))$ is differentiable almost everywhere and, for a.e. $t>0$,
    \begin{equation*}
        \frac{d}{dt}\mu(t,\bar x(t)) = \p_t \mu(t, \bar x(t)).
    \end{equation*}
\end{lemma}
\begin{proof}[Proof of Proposition \ref{prop:instantbounds}]  
\textbf{Mass conservation and positivity.} 
Mass conservation is obtained after integrating the equation in space and using the periodic boundary conditions. The nonnegativity can be obtained from the transport equation structure. For instance one can multiply the equation by $\chi_{u_{\eps}<0}$ (up to smoothing and sending the regularization to 0) and integrate in space.
\bigbreak
\textbf{Weak maximum principle.} Let us prove that, if $u_{0,\ep} \le c$ for some $c\ge 0$, then $u_\ep(t)\le c$ for all $t\ge 0$. In the sequel, we denote $\bar u_\ep$ the total mass of $u_\ep$ (which is conserved). We have
    \begin{align}
        \frac{d}{dt}\int_{\T^d} (u_\ep(t,x) - c)_+dx &=  \int_{u_\ep(t) \ge c} (\div(u_\ep^m\nab\g\ast u_\ep)(t,x) + \ep \D u_\ep(t,x))dx\nonumber \\
        &= \int_{u_\ep\ge c} \nab u_\ep^m(t,x)\cdot\nab\g\ast u_\ep (t,x) dx \label{term1WMP}\\
        &- \int_{u_\ep \ge c} u_\ep(t,x)^m(u_\ep(t,x) - \bar u_\ep) dx + \ep \int_{u_\ep\ge c} \D u_\ep(t,x) dx.\nonumber
    \end{align}
    First, notice that $\nab (u_\ep^m -c^m)_+ = \indic_{u_\ep \ge c} \nab u_\ep^m$ almost everywhere on $\T^d$, so that \eqref{term1WMP} can be rewritten, after integrating by parts, as
    \begin{equation*}
        \int_{\T^d} \nab (u_\ep^m -c^m)_+(t,x) \cdot \nab\g\ast u_\ep(t,x)dx = \int_{\T^d} (u_\ep^m -c^m)_+(t,x) (u_\ep(t,x)-\bar u_\ep )dx.
    \end{equation*}
    Moreover, we have in the sense of distributions
    \begin{equation*}
        \D (u_\ep - c)_+= |\nab u_\ep|^2 \d_{u_\ep = c} + \indic_{u_\ep \ge c} \D u_\ep,
    \end{equation*}
    so that the viscous term is negative and can be dropped. Overall, we obtain
    \begin{equation*}
        \frac{d}{dt}\int_{\T^d} (u_\ep(t,x) - c)_+dx \le -c^m\int_{u_\ep\ge c} (u_\ep(t,x) -\bar u_\ep )dx \le 0,
    \end{equation*}
    where the last inequality is obtained if $ \bar u_\ep \le c$. Finally, we have obtained that, if $u_{0,\ep} \le c$ on $\T^d$, then in particular $\bar u_\ep \le c$ and so $u_\ep(t)\le c$ for all $t\ge 0$. The same argument proves that if $u_{0,\ep}\ge c$ for some $c\ge 0$, then $u_\ep(t)\ge c$ for all $t\ge 0$.

    This also implies, by the semigroup property, that $t\mapsto\min u_\ep(t)$ and $t\mapsto \max u_\ep(t)$ are respectively nondecreasing and nonincreasing functions.

\bigbreak
\textbf{Estimate on the maximum.} Let $t\mapsto \bar x(t)$ be a (not necessarily smooth) curve of maximum points to $u_\ep$. Evaluating the equation at $(t,\bar x(t))$ gives, thanks to Lemma \ref{lem:mincurve},
    \begin{align*}
        \frac{d}{dt}\| u_\ep(t)\|_{L^\infty} &= \p_t u_\ep(t,\bar x(t))\\
        &= \nab u_\ep^m(t,\bar x(t))\cdot\nab\g\ast u_\ep(t,\bar x(t)) - u_\ep(t,\bar x(t))^m(u_\ep(t,\bar x(t)) - \bar u_\ep ) \\
        &+\ep \D u_\ep(t,\bar x(t)).
    \end{align*}
    Using first and second order conditions at maximum points, we obtain
    \begin{equation*}
         \frac{d}{dt}\|u_\ep(t)\|_{L^\infty} + \| u_\ep(t)\|_{L^\infty}^{m+1} \le \|u_\ep(t)\|_{L^\infty}^m\bar u_\ep.
    \end{equation*}
    We can then conclude that for all $t>0$, $\|u_\ep(t)\|_{L^\infty} \le \Phi_{\|u_{0,\ep}\|_{L^\infty}}(t)$, where $\Phi_{\|u_{0,\ep}\|_{L^\infty}}$ is given in Lemma \ref{eq:ODE}.
    
\bigbreak
\textbf{Estimate on the minimum.} Let $t\mapsto \bar x(t)$ be a (not necessarily smooth) curve of minimum points to $u_\ep$. Evaluating the equation at $(t,\bar x(t))$ gives, thanks to Lemma \ref{lem:mincurve},
    \begin{align*}
        \frac{d}{dt}\min u_\ep(t) &= \p_t u_\ep(t,\bar x(t))\\
        &= \nab u_\ep^m(t,\bar x(t))\cdot\nab\g\ast u_\ep(t,\bar x(t)) - u_\ep(t,\bar x(t))^m(u_\ep(t,\bar x(t)) - \bar u_\ep) \\
        &+\ep \D u_\ep(t,\bar x(t)).
    \end{align*}
    Using first and second order conditions at minimum points, we obtain
    \begin{equation*}
         \frac{d}{dt}\min u_\ep(t) + \min u_\ep(t)^{m+1} \ge \min u_\ep(t)^m\bar u_\ep.
    \end{equation*}
    We can then conclude that for all $t>0$, $\min u_\ep(t) \ge \Phi_{\min u_{0,\ep}}(t)$.
    
\bigbreak
\textbf{Decreasing of the energy and $L^p$ norms.} Multiplying the equation by $g\ast u_{\eps}$, integrating in space and using the symmetry of the kernel gives
\begin{align*}
    \frac{d}{dt}\hal\int_{\T^d}\g\ast u_\ep(t,x) u_\ep(t,x) dx &=  -\int_{\T^d} |\nab \g \ast u_\ep(t,x)|^2 u_\ep(t,x)^m dx \\
    &+ \ep \int_{\T^d} \g\ast u_\ep(t,x) \D u_\ep(t,x) dx.
\end{align*}
The viscous term can be rewritten, integrating by parts, as
\begin{align*}
    \ep \int_{\T^d} \g\ast u_\ep(t,x) \D u_\ep(t,x) dx &= -\ep \int_{\T^d} u_\ep(t,x) (u_\ep(t,x) - \bar u_\ep) dx  \\
    &= -\ep \bigg(\int_{\T^d} u_\ep^2dx - \bar u_\ep^2\bigg) \\
    &\le 0,
\end{align*}
where the last estimate follows from Jensen's inequality. Finally, for any $1\le p <\infty$, we multiply the equation by $p u_{\eps}^{p-1}$ and integrate in space to obtain
\begin{align*}
    \frac{d}{dt}\int_{\T^d} u_\ep^p dx &= \int_{\T^d}pu_\ep^{p-1} \div(u_\ep^m\nab\g\ast u_\ep)dx+ p\ep \int_{\T^d} u_\ep^{p-1}\D u_\ep dx \\
    &= -\frac{p(p-1)}{p+m-1} \int_{\T^d} \nab u_\ep^{p+m-1}\cdot\nab\g\ast u_\ep dx \\
    &-p(p-1)\ep \int_{\T^d} u_\ep^{p-2}|\nab u_\ep|^2 dx \\
    &\le-\frac{p(p-1)}{p+m-1}\int_{\T^d} u_\ep^{p+m-1}(u_\ep -\bar u_\ep) dx.
\end{align*}
Notice that, by Holder's inequality,
\begin{equation}\label{ineq:Jensen}
    \int_{\T^d} u_\ep^{p+m-1} dx \le \int_{\T^d} u_\ep^{p+m}dx\bigg(\int_{\T^d}u_\ep^{p+m}dx\bigg)^{-\frac{1}{p+m}},
\end{equation}
and $p+m >1$ so that by Jensen's inequality $\bigg(\int_{\T^d}u_\ep^{p+m}dx\bigg)^{\frac{1}{p+m}}\ge  \int_{\T^d} u_\ep dx = \bar u_\ep$. Overall,
\begin{equation*}
     \frac{d}{dt}\int_{\T^d} u_\ep^p dx \le 0.
\end{equation*}
\end{proof}

\begin{proof}[Proof of Lemma \ref{eq:ODE}]
    First, consider the case $\be >\bar u_\ep $. Using the notation $\Psi_\be :=  \Phi_\be -\bar u_\ep$, we have both $\Psi_\be \ge 0$ and
    \begin{equation*}
        \dot\Psi_\be = \Phi^m_\be (\bar u_\ep - \Phi_\be) = -(\bar u_\ep + \Psi_\be )^m\Psi_\be.
    \end{equation*}
    This can be bounded either by $-\bar u_\ep^m \Psi_\be$ or $-\Psi_\be^{m+1}$. Therefore, $\Phi_\be(t) \le \bar u_\ep + (\be - \bar u_\ep) e^{-\bar u_\ep^m t}$ and also $\Phi_\be (t) \le\bar u_\ep +  (tm + (\be -\bar u_\ep)^{-m})^{-1/m}$. 

    Now, consider the case $\be <\bar u_\ep$. Define $\Psi_\be := -\Phi_\be + \bar u_\ep$, so that $\Psi_\be\ge 0$ and
    \begin{equation*}
        \dot\Psi_\be = - \Phi_\be^m(\bar u_\ep - \Phi_\be) = -\Phi_\be^m\Psi_\be \le -\be^m \Psi_\be.
    \end{equation*}
    Therefore, $\Phi_\be(t) \ge\bar u_\ep - (\bar u_\ep -\be ) e^{-\be^mt}$.

    If $m<1$ and $\be = 0$, there is no unique solution. In fact, $\Phi_0 \equiv 0$ is a solution, but the ``force field'' is not Lipschitz continuous. Define, for all $a\in (0,1)$,
    \begin{equation*}
        \tau_a := \inf\{t\ge 0 : \Phi_\be (t) \ge a\bar u_\ep\}.
    \end{equation*}
    On $[0,\tau_a]$, $\Phi_\be \le a\bar u_\ep$ and 
    \begin{equation*}
        \dot \Phi_\be = \Phi_\be^m(\bar u_\ep - \Phi_\be)\ge \Phi_\be^m  (1-a) \bar u_\ep.
    \end{equation*}
    This implies $\Phi_\be(t) \ge (\be^{1-m} + \bar u_\ep(1-a) t)^{1/(1-m)}$, for all $t\in [0,\tau_a]$. Using the same argument as before, but starting at $t=\tau_a$ instead of $t=0$, we obtain for all $t\in [\tau_a, \infty)$,
    \begin{equation*}
        \Phi_\be(t) \ge \bar u_\ep - \bar u_\ep (1- a) e^{-a^m\bar u_\ep^m(t-\tau_a)}.
    \end{equation*}
    This is true for all $a\in (0,1)$. By construction, we have $\Phi_\be(\tau_a) = a\bar u_\ep \ge (\bar u_\ep (1-a)\tau_a )^{1/(1-m)}$, so that 
    \begin{equation*}
        \tau_a \le \frac{a^{1-m}}{\bar u_\ep^m (1-a)}.
    \end{equation*}
    Replacing $\tau_a$ above, one obtains for $t\in [\tau_a, \infty)$,
    \begin{equation*}
        \Phi_\be(t)\ge \bar u_\ep - \bar u_\ep (1-a)e^{\frac{a}{1-a}} e^{-a^m\bar u_\ep^m t}.
    \end{equation*}
    Take $a=\hal$ to conclude.
\end{proof}

\subsection{A compactness result}\label{subsec:compactness}

The method we use in order to obtain strong-$L^1$ compactness has been developed in~\cite{belgacemjabin2013}. The idea is that a sequence $(u_\ep)_\ep$ of $L^1(\T^d)$ functions is precompact in $L^1(\T^d)$ if and only if
\begin{equation*}
    \lim_{h\to 0}\frac{1}{|\log h|} \limsup_{\ep\to 0} \iint_{\T^{2d}} K_h(x-y) |u_\ep(x)-u_\ep(y)|dxdy = 0,
\end{equation*}
where we define the kernel $K_h:\T^d\to \R$ by
\begin{equation*}
        K_h(x) = \frac{1}{||x|+h|^d} + \Gamma(x),
\end{equation*}
for some smooth $\Gamma :\T^d\to \R$ such that $K_h$ is a periodic function.
We record the following lemma, taken from \cite[Proposition 4.1]{belgacemjabin2013}. We notice that we state this lemma on the torus, whereas it was previously obtained on the Euclidean setting. Nevertheless, a carefull examination of the proof shows that it can be extended on the periodic domain.
\begin{lemma}\label{ineq:CZO}
    Let $1<p<\infty$, $a\in C^\infty(\T^d,\T^d)$ and $f\in C^\infty(\T^d)$. There exists $C>0$ depending on $p$ such that
    \begin{multline*}
        \int_{\T^d\times \T^d}\nab K_h(x-y) \cdot (a(x)-a(y)) |f(x)-f(y)|^2 dxdy \\
        \le C \big(\|\div a\|_{L^\infty} +\|\nab a\|_{L^p}\big)\iint_{\T^d\times \T^d} K_h(x-y) |f(x)-f(y)|^2dxdy \\
        +C \|f\|_{L^\infty} \|f\|_{L^{p^*}} \|\nab a\|_{L^p}|\log h|^{1/\bar p},
    \end{multline*}
    where $ 1/p+1/{p^*}=1$ and $\bar p = \min (p,2)$.
\end{lemma}
With this lemma at hand, we can now prove the following result.
\begin{prop}\label{prop:compactness}
    Denote $u_\ep$ the unique solution to \eqref{eq:viscousPDE}. Assume further that $u_{0,\ep}\ge c>0$ if $m<1$. We have, for all $t>0$,
    \begin{multline*}
        \lim_{h\to 0}\frac{1}{|\log h|}\limsup_{\ep \to 0} \iint_{\T^d\times\T^d}K_h(x-y) |u_\ep(t,x) - u_\ep(t,y)|dxdy \\
        \le C(t) \lim_{h\to 0}\frac{1}{|\log h|}\limsup_{\ep \to 0} \iint_{\T^d\times\T^d}K_h(x-y) |u_{0,\ep}(x) - u_{0,\ep}(y)|dxdy,
    \end{multline*}
    where $C(t)$ is explicit, and depends on the a priori bounds on $u_\ep$. In particular, if $(u_{0,\ep})_\ep$ is compact in $L^1(\T^d)$, so is $(u_\ep(t))_\ep$, for all $t>0$.
\end{prop}
\begin{proof}
    Let $u_\ep$ be the unique solution to \eqref{eq:viscousPDE}. Define
    \begin{equation*}
        Q_\ep(t) := \iint_{\T^d\times \T^d} K_h(x-y)|u_\ep(t,x)-u_\ep(t,y)| dxdy.
    \end{equation*}
    We start from Kr\"uzhkov's method of doubling variables~\cite{MR267257} and we obtain
    \begin{multline*}
        \p_t|u_\ep(x)-u_\ep(y)| + |u_\ep(x)^{m+1}-u_\ep(y)^{m+1}| - (u_\ep(x)+u_\ep(y))|u_\ep(x)^m-u_\ep(y)^m| \\
        -\div_x (|u_\ep(x)^m-u_\ep(y)^m| \nab\g\ast u_\ep)(x) - \div_y(|u_\ep(x)^m-u_\ep(y)^m|\nab\g\ast u_\ep)(y) \\
        -\ep (\D_x+\D_y)|u_\ep(x)-u_\ep(y)|\le 0.
    \end{multline*}
    Notice that
    \begin{multline*}
        |u_\ep(x)^{m+1}-u_\ep(y)^{m+1}| - (u_\ep(x)+u_\ep(y))|u_\ep(x)^m-u_\ep(y)^m| \\
        = \sgn(u_\ep(x)-u_\ep(y)) u_\ep(x)u_\ep(y)(u_\ep(y)^{m-1}-u_\ep(x)^{m-1}).
    \end{multline*}
    Therefore, when $m<1$ the first terms yield a nonnegative contribution which can be discarded. Multiplying by $K_h$, integrating in space and perfomring integration by parts we are thus left with
    \begin{align*}
        \frac{d}{dt}Q_\ep 
        &\le -\iint_{\T^d\times\T^d} \nab K_h(x-y)\cdot \big(\nab\g\ast u_\ep(x) - \nab\g\ast u_\ep(y)\big)  |u_\ep(x)^m-u_\ep(y)^m| dxdy\\
        &+ \indic_{m>1}\iint_{\T^d\times \T^d} K_h(x-y) (u_\ep(x)+u_\ep(y)) |u_\ep(x)^m-u_\ep(y)^m|dxdy \\
        &+ 2\ep \iint_{\T^d\times\T^d} \D K_h(x-y)|u_\ep(x)-u_\ep(y)|dxdy. 
    \end{align*}
    Let us first deal with the viscous term, which can be bounded by
    \begin{equation}\label{term:viscous}
        C\frac{\ep}{h^2} \iint_{\T^d\times\T^d}K_h(x-y) |u_\ep(x)-u_\ep(y)| dxdy.
    \end{equation}
    The second term is then bounded, using $|a^m-b^m|\le m(a^{m-1}+b^{m-1})|a-b|$, by
    \begin{equation}\label{term:second}
        C_m\indic_{m>1} \|u_\ep\|_{L^\infty}^m \iint_{\T^d\times\T^d} K_h(x-y) |u_\ep(x)-u_\ep(y)|dxdy.
    \end{equation}
    In order to deal with the first term, we introduce
    \begin{equation*}
        \chi_\ep(t,x,\xi) := \indic_{\xi \le u_\ep(t,x)},
    \end{equation*}
    so that
    \begin{equation*}
        |u_\ep(x)^m-u_\ep(y)^m| = \int_0^\infty m\xi^{m-1} |\chi_\ep(t,x,\xi)-\chi_\ep(t,y,\xi)|^2d\xi.
    \end{equation*}
    Hence,
    \begin{multline*}
        -\iint_{\T^d\times\T^d} \nab K_h(x-y)\cdot \big(\nab\g\ast u_\ep(x) - \nab\g\ast u_\ep(y)\big)  |u_\ep(x)^m-u_\ep(y)^m| dxdy\\
        = -m\int_0^\infty d\xi \ \xi^{m-1}\iint_{\T^d\times\T^d} \nab K_h(x-y)\cdot \big(\nab\g\ast u_\ep(x) - \nab\g\ast u_\ep(y)\big)  |\chi_\ep(x)-\chi_\ep(y)|^2dxdy.
    \end{multline*}
    We now use Lemma \ref{ineq:CZO} to bound this by
    \begin{multline}\label{ineq:CZOtmp}
        C(\|u_\ep\|_{L^\infty} + \|\nab^{\otimes 2}\g\ast u_\ep\|_{L^p})\int_0^\infty d\xi \ \xi^{m-1} \iint_{\T^d\times \T^d}K_h(x-y) |\chi_\ep(x,\xi)-\chi_\ep(y,\xi)|^2 dxdy  \\
        + C |\log h|^{1/\bar p} \|\nab^{\otimes 2}\g\ast u_\ep\|_{L^p} \int_0^\infty d\xi \ \xi^{m-1} \|\chi_\ep(\xi)\|_{L^{p^*}_x},
    \end{multline}
    for any $1<p<\infty$, $ 1/{p}+1/{p^*} = 1$, $\bar p = \min(p,2)$, and where $C>0$ depends additionally on $m$. Moreover, since $L^{p^*}(\T^d)\subset L^1(\T^d)$, we have
    \begin{align}\label{ineq:chiLp}
        \int_0^\infty d\xi \ \xi^{m-1} \|\chi_\ep(\xi)\|_{L^{p^*}_x} &\le 2\int_0^\infty d\xi \ \xi^{m-1} |\{u_\ep\ge \xi\}|^{1/p^*} \nn \\
        &=2\int_0^1 d\xi\ \xi^{m-1}|\{ u_\ep\ge\xi \}|^{1/p^*} +2\int_1^\infty d\xi\ \xi^{m-1}|\{ u_\ep\ge\xi \}|^{1/p^*}.
    \end{align}
    Using Markov's inequality, we bound this by
        $C_q \|u_\ep\|_{L^q}^{1/p^*} + C_{q'}\|u_\ep\|_{L^{q'}}^{1/p^*}$,
    for some $1\le q< mp^*<q'$, and the constants $C_q,C_{q'}$ blow up when $q,q'\to mp^*$, respectively. Finally, we get that \eqref{ineq:CZOtmp} is bounded by
    \begin{multline*}
        C(\|u_\ep\|_{L^\infty} + \|\nab^{\otimes 2}\g\ast u_\ep\|_{L^p}) \iint_{\T^d\times \T^d}K_h(x-y) |u_\ep(x)^m-u_\ep(y)^m| dxdy  \\
        + C_{q,q',p,m} |\log h|^{1/\bar p} \|\nab^{\otimes 2}\g\ast u_\ep\|_{L^p} \|u_\ep\|_{L^q\cap L^{q'}}^{1/p^*}.
    \end{multline*}
    We then conclude using $|a^m-b^m|\le m(a^{m-1}+b^{m-1})|a-b|$ and collecting \eqref{term:viscous} and \eqref{term:second} that
    \begin{multline*}
        \frac{d}{dt}Q_\ep \le C(\|u_\ep\|_{L^\infty} +  \|\nab^{\otimes 2}\g\ast u_\ep\|_{L^p}) (\|u_\ep\|_{L^\infty}^{m-1}\indic_{m>1} + (\inf u_\ep)^{m-1}\indic_{m<1}) Q_\ep \\
        + C_{q,q',p,m} |\log h|^{1/\bar p} \|\nab^{\otimes 2}\g\ast u_\ep\|_{L^p} \|u_\ep\|_{L^q\cap L^{q'}}^{1/p^*} + C_m \indic_{m>1} \|u_\ep\|_{L^\infty}^m Q_\ep + C\frac{\ep}{h^2} Q_\ep.
    \end{multline*}
    Using Gr\"onwall's lemma gives
    \begin{multline*}
        Q_\ep(t) \le \bigg(Q_\ep(0) + C_{q,q',p,m} |\log h|^{1/\bar p} \int_0^t \|\nab^{\otimes 2}\g\ast u_\ep\|_{L^p} \|u_\ep\|_{L^q\cap L^{q'}}^{1/p^*}dt\bigg) \\
        \times \exp\int_0^t \bigg(C(\|u_\ep\|_{L^\infty} +  \|\nab^{\otimes 2}\g\ast u_\ep\|_{L^p}) (\|u_\ep\|_{L^\infty}^{m-1}\indic_{m>1} + (\inf u_\ep)^{m-1}\indic_{m<1}) +C_m \indic_{m>1} \|u_\ep\|_{L^\infty}^m  + C\frac{\ep}{h^2}   \bigg)  dt.
    \end{multline*}
    Calderón–Zygmund theory shows that $\|\nab^{\otimes 2}\g\ast u_\ep\|_{L^p}$ can be bounded by $\|u_{\eps}\|_{L^{p}}$ for $1<p<+\infty$. Dividing by $\log h$ and usign estimates from Proposition~\ref{prop:instantbounds} yields the result.
\end{proof}

With strong compactness at hand, we can now prove the existence of a weak solution, that is Theorem~\ref{thm:existence}.

\begin{proof}[Proof of Theorem~\ref{thm:existence}]

Let $T>0$. Consider $\varphi \in C_c^\infty([0,T)\times\T^d)$. We have, 
\begin{multline*}
    -\int_0^T\int_{\T^d} \p_t\varphi(t,x) u_\ep(t,x) dxdt - \int_{\T^d} \varphi_0(x)u_{0,\ep}(x)dx \\
    + \int_0^T\int_{\T^d} \nab\varphi(t,x)\cdot\nab\g\ast u_\ep(t,x) u_\ep(t,x)^m dxdt 
    = \ep \int_0^T\int_{\T^d} \D\varphi(t,x) u_\ep(t,x)dxdt.
\end{multline*}
We now collect the uniform bounds obtained regarding the solution $u_\ep$:
\begin{enumerate}[label=(\roman*).]
    \item For all $1\le p \le \infty$, there exists a constant $C>0$ such that for any $t>0$, $\|u_\ep(t)\|_{L^p} \le C$,
    \item If $m<1$, there exists $c>0$ such that for all $t\ge 0$, $\min u_\ep(t)\ge c>0$.
\end{enumerate}
Passing to the limit in linear terms follows by weak compactness, induced by the $L^{p}$ bounds and the Banach-Alaoglu theorem. The diffusion term converges to 0 as $u_{\eps}$ is uniformly bounded in $L^{1}((0,T)\times\T^d)$.
Proposition \ref{prop:compactness} implies that, for a.e. $t>0$, there exists $u(t) \in  L^\infty$ such that $u_\ep(t)\to u(t)$, strongly in $L^1$. Since $u(t)$ is bounded for a.e. $t>0$, this convergence holds for all $L^p$, $1\le p<\infty$ by dominated convergence theorem. We also note that when $m<1$, it holds for a.e. $t>0$,
\begin{equation*}
    \int_{\T^d} |u_\ep(t)^m-u(t)^m|dx \le 2mc^{m-1} \|u_\ep(t)-u(t)\|_{L^1}\xrightarrow[\ep\to 0]{}0.
\end{equation*}
From this, we deduce $u_\ep(t)^m\to u(t)^m$ strongly in any $L^p$, $1\le p < \infty$, for a.e. $t>0$.
Finally, $\nab\g\ast u_\ep(t) \xrightarrow[\ep\to 0]{}\nab\g\ast u(t)$ strongly in any $L^p$, $1\le p\le \infty$.  By dominated convergence, we can pass to the limit in the nonlinear term since
$\nab\g\ast u_\ep(t)u_\ep(t)^m \xrightarrow[\ep\to 0]{} \nab\g\ast u(t) u(t)^m$ strongly in $L^p$, $1\le p<\infty$, for a.e. $t>0$. One obtains
\begin{multline*}
    -\int_0^T\int_{\T^d} \p_t\varphi(t,x) u(t,x) dxdt - \int_{\T^d} \varphi_0(x)u_{0}(x)dx \\
    + \int_0^T\int_{\T^d} \nab\varphi(t,x)\cdot\nab\g\ast u(t,x) u(t,x)^m dxdt=0.
\end{multline*}
This equation remains valid by density for all $\varphi \in C_c^{1}([0,T)\times\T^d)$.
Finally, for any $\eta\in C^2$ convex and $q'(\xi) := m\xi^{m-1} \eta'(\xi)$, one has
\begin{align*}
    \p_t \eta(u_\ep) &= \eta'(u_\ep) \div(u_\ep^m\nab\g\ast u_\ep) + \ep \eta'(u_\ep) \D u_\ep \\
    &= q'(u_\ep) \nab u_\ep\cdot\nab\g\ast u_\ep + \eta'(u_\ep) u_\ep^m \D\g\ast u_\ep  + \ep \eta'(u_\ep) \D u_\ep .
\end{align*}
Using that $\D \eta(u_\ep ) = \eta ''(u_\ep) |\nab u_\ep|^2 + \eta'(u_\ep) \D \eta (u_\ep)$ and \eqref{eq:green_kernel}, one obtains
\begin{align*}
     \p_t \eta(u_\ep) \le \div (q(u_\ep)\nab\g\ast u_\ep) + (q(u_\ep) - \eta'(u_\ep) u_\ep^m ) (u_\ep-\bar u_\ep) + \ep \D \eta (u_\ep).
\end{align*}
Sending again $\ep\to 0$ using our strong convergence result concludes the proof.
\end{proof}

\subsection{Weak-strong uniqueness}

The uniqueness of non-radial entropy solutions to \eqref{eq:PDE} is an open question in dimension $d\ge 2$. In this subsection, we provide a weak-strong uniqueness result: provided that there exists an entropy solution $u$ satisfying $\nab u\in L^1([0,T], L^{d,1}(\T^d))$, it is the unique entropy solution to \eqref{eq:PDE} on $[0,T]$. Note that this regularity implies that the solution is continuous, a property that shall be propagated only for short times. Nevertheless, this weak-strong uniqueness principle is a minimal requirement to our problem as far as uniqueness of entropy solutions remains open.

Let us now prove Theorem~\ref{thm:weak_strong_uniqueness}. We first record the following lemma, resulting from Young's inequality for convolutions on Lorentz spaces (also known as O'Neil inequality~\cite{MR146673}):

\begin{lemma}\label{lem:O'Neil}
There exists $C>0$ such that for all $u,v \in L^{1}(\T^d)$ with $\nabla \g\ast u$, $\nabla \g\ast v \in L^{\frac{d}{d-1},\infty}$, we have
 $$
 \|\nabla \g\ast (u-v)\|_{L^{\frac{d}{d-1},\infty}}\le C\|u-v\|_{L^{1}}.
 $$
\end{lemma}

\begin{proof}[Proof of Theorem~\ref{thm:weak_strong_uniqueness}] 

We let $u,v$ be two entropy solutions of~\eqref{eq:PDE} with initial condition $u_0$, $v_0$, where $v$ has higher regularity assumed in Theorem~\ref{thm:weak_strong_uniqueness}. Taking the differences of the equation on $u$ and $v$ and multiplying by $\sgn(u-v)$ (this can be made rigorous thanks to Kr\"uzhkov's argument): 
\begin{align*}
    \p_t |u- v| &\le \div(|u^m-v^m|\nab\g\ast v)  
    +\sgn (u-v)\nab u^m\cdot\nab\g\ast (u-v) - u^m |u-v|.
\end{align*}
Integrating in space, we obtain
\begin{align*}
    \frac{d}{dt}\| u-v\|_{L^1} &\le \int_{\T^d} |\nab u^m||\nab\g\ast (u-v)|dx \\
    &\le \|\nab u^m\|_{L^{d,1}}\|\nab\g\ast (u-v)\|_{L^{\frac{d}{d-1},\infty}} \\
    &\le  C\|\nab u^m\|_{L^{d,1}}\| u-v\|_{L^1}.
\end{align*}
We classically conclude by using Gr\"onwall's lemma.

\end{proof}

\subsection{Local well-posedness}

We end this section by showing that the problem \eqref{eq:PDE} is locally well-posed, assuming that the initial condition is Lipschitz continuous. 

\begin{proof}[Proof of Proposition~\ref{prop:strong_solution}]
We work in the approximated scheme and then send $\varepsilon\to 0$. We consider $u_{0,\ep}\ge c>0$, but this condition can be removed if $m\ge2$ as it will be clear from the proof.  Let $u_\ep$ be the unique solution to \eqref{eq:viscousPDE} with initial datum $u_{0,\ep}$. We have
    \begin{equation*}
        \p_t u_\ep = mu_\ep^{m-1}\nab u_\ep\cdot\nab\g \ast u_\ep - u_\ep^m(u_\ep - \bar u_\ep) + \ep \D u_\ep.
    \end{equation*}
    Thus,
    \begin{align*}
        \p_t \nab u_\ep &= m\nab u_\ep^{m-1}\nab u_\ep\cdot\nab\g \ast u_\ep + mu_\ep^{m-1} \nab^{\otimes 2}u_\ep \cdot\nab\g\ast u_\ep \\
        &+mu_\ep^{m-1}\nab u_\ep\cdot\nab^{\otimes 2}\g\ast u_\ep - (m+1) u_\ep^m \nab u_\ep + m\bar u_\ep u_\ep^{m-1}\nab u_\ep   + \ep \D \nab u_\ep.
    \end{align*}
    Multiplying the equation by $\displaystyle p|\nab u_\ep|^{p-1}\frac{\nab u_\ep}{|\nab u_\ep|}$, one obtains
    \begin{align}\label{dtdxu}
         \p_t |\nab u_\ep |^p &= mp|\nab u_\ep|^{p-1}\frac{\nab u_\ep\cdot \nab u_\ep^{m-1}}{|\nab u_\ep|} \nab u_\ep\cdot\nab \g\ast u_\ep + mp |\nab u_\ep|^{p-1} u_\ep^{m-1} \nab |\nab u_\ep|\cdot \nab\g\ast u_\ep \nn \\
        & + mp|\nab u_\ep|^{p-1} u_\ep^{m-1} \nab^{\otimes 2}\g\ast u_\ep : \frac{\nab u_\ep \otimes \nab u_\ep}{|\nab u_\ep|} - (m+1)pu_\ep^m|\nab u_\ep|^{p} \nn \\
        & + mp\bar u_\ep u_\ep^{m-1} |\nab u_\ep |^p + \ep p|\nab u_\ep|^{p-1}\frac{\nab u_\ep}{|\nab u_\ep |}\cdot \nab \D u_\ep .
    \end{align}
    Notice that the last term, once integrated, gives
    \begin{align*}
    &\ep p \int_{\T^d} |\nab u_\ep|^{p-1}\frac{\nab u_\ep}{|\nab u_\ep |}\cdot \nab \D u_\ep dx \\
    &= -\ep p \int_{\T^d}|\nab u_\ep|^{p-2} |\D u_\ep|^2 dx - \ep p(p-2) \int_{\T^d} |\nab u_\ep|^{p-4} \nab u_\ep\otimes \nab u_\ep : \nab^{\otimes 2}u_\ep   \D u_\ep dx \\
    &=  -\ep p \int_{\T^d}|\nab u_\ep|^{p-2} |\D u_\ep|^2 dx - \ep p(p-2) \int_{\T^d} |\nab u_\ep|^{p-3} \nab u_\ep\cdot \nab |\nab u_\ep|   \D u_\ep dx\\
    &=  -\ep p \int_{\T^d}|\nab u_\ep|^{p-2} |\D u_\ep|^2 dx - \ep  \int_{\T^d}\nab u_\ep\cdot \nab |\nab u_\ep|^{p-2}   \D u_\ep dx \\
    &=  -\ep (p-1) \int_{\T^d}|\nab u_\ep|^{p-2} |\D u_\ep|^2 dx  +\ep \int_{\T^d} |\nab u_\ep|^{p-2}\nab u_\ep \cdot \nab \D u_\ep dx.
    \end{align*}
    Therefore, the viscous term in \eqref{dtdxu} reads, once integrated,
    \begin{equation*}
        -\ep p  \int_{\T^d}|\nab u_\ep|^{p-2} |\D u_\ep|^2 dx \le 0.
    \end{equation*} 
    Integrating in space and by parts thus gives
    \begin{align*}
        &\frac{d}{dt}\|\nab u_\ep \|_{L^p}^p \\
        &\le mp\int_{\T^d} |\nab u_\ep|^{p-2}\nab u_\ep\cdot \nab u_\ep^{m-1} \nab u_\ep\cdot \nab\g\ast u_\ep dx  -m\int_{\T^d} \nab u^{m-1}_\ep\cdot\nab\g\ast u_\ep |\nab u_\ep|^p dx \\
        &+ m\int_{\T^d} u_\ep^{m-1}(u_\ep-\bar u_\ep) |\nab u_\ep|^p dx + mp \int_{\T^d} |\nab u_\ep|^{p-2} u_\ep^{m-1}\nab^{\otimes 2}\g\ast u_\ep : \nab u_\ep\otimes \nab u_\ep dx \\
        &-mp\int_{\T^d} u_\ep^{m-1}(u_\ep - \bar u_\ep) |\nab u_\ep|^p dx - p \int_{\T^d} u_\ep^m |\nab u_\ep|^p dx \\
        &=: I + II+ III+ IV + V + VI.
    \end{align*}
    We then combine and bound the differents terms above as follows:
    \begin{align}\label{ineq:I-VI}
        I +II &= m(p-1) \int_{\T^d} |\nab u_\ep|^p \nab u_\ep^{m-1}\cdot\nab \g\ast u_\ep dx \\ &\le m(p-1) \|\nab u_\ep^{m-1}\|_{L^\infty} \|\nab \g\ast u_\ep\|_{L^\infty} \|\nab u_\ep\|_{L^p}^p \nonumber \\
        III + V+ VI &= (m-mp-p)\int_{\T^d} u_\ep^{m-1}(u_\ep - \bar u_\ep) |\nab u_\ep|^p dx - p\bar u_\ep \int_{\T^d} u_\ep^{m-1} |\nab u_\ep|^p dx \nonumber\\
        IV  &\le mp \|\nab^{\otimes 2}\g\ast u_\ep\|_{L^\infty} \| u_\ep^{m-1}\|_{L^\infty} \|\nab u_\ep\|_{L^p}^p.\nonumber
    \end{align}
    Inserting these estimates in \eqref{dtdxu} and using that the $L^p$ norms decrease along the flow and $\nab \g \in L^1(\T^d)$,
    \begin{align*}
        \frac{d}{dt}\| \nab u_\ep \|_{L^p}^p &\le (p-1) C_{m,d} (c^{m-2}+\|u_{0,\ep}\|_{L^\infty}^{m-2})\|u_{0,\ep}\|_{L^\infty} \| \nab u_\ep \|_{L^\infty} \|\nab u_\ep\|_{L^p}^p \\
        & + pC_{m} (\|u_{0,\ep}\|_{L^\infty}^{m-1}+c^{m-1}) \| u_{0,\ep}\|_{L^\infty} \|\nab u_\ep\|_{L^p}^p   \\
        & + pC_{m,d} (c^{m-1}+\| u_{0,\ep}\|_{L^\infty}^{m-1}) \|\nab u_\ep\|_{L^\infty} \|\nab u_\ep\|_{L^p}^p. 
    \end{align*}
    Therefore, 
    \begin{multline*}
        \frac{d}{dt}\log \|\nab u_\ep\|_{L^p} \le C_{m,d} \big((c^{m-2}+\|u_{0,\ep}\|_{L^\infty}^{m-2})\|u_{0,\ep}\|_{L^\infty} + c^{m-1}+\|u_{0,\ep}\|_{L^\infty}^{m-1} \big) \|\nab u_\ep\|_{L^\infty} \\
        + C_m (\| u_{0,\ep}\|_{L^\infty}^{m-1}+c^{m-1}) \| u_{0,\ep}\|_{L^\infty }.
    \end{multline*}
    Sending $p\to\infty$ above gives
    \begin{multline*}
        \frac{d}{dt}\| \nab u_\ep\|_{L^\infty}\le C_{m,d} \big((c^{m-2}+\|u_{0,\ep}\|_{L^\infty}^{m-2})\|u_{0,\ep}\|_{L^\infty} + c^{m-1}+\|u_{0,\ep}\|_{L^\infty}^{m-1} \big) \|\nab u_\ep\|_{L^\infty}^2 \\
        + C_m (\| u_{0,\ep}\|_{L^\infty}^{m-1}+c^{m-1}) \| u_{0,\ep}\|_{L^\infty }\|\nab u_\ep\|_{L^\infty}.
    \end{multline*}
Using Gr\"onwall's lemma, we obtain an estimate on $\|\nabla u_{\eps}(t)\|_{L^{\infty}}$. Sending $\varepsilon\to 0$ then yields the result by lower semi-continuity of the norm. 
\end{proof}

\begin{remark}
Taking a closer look at the proof, we find that $T_*$ can be estimated from below, and $\|\nabla u(t)\|_{L^{\infty}}$ can be estimated from above. Specifically,
 \begin{equation*}
    T_*\ge C_{m,d}^{-1} \big((c^{m-2}+\|u_{0,\ep}\|_{L^\infty}^{m-2})\|u_{0,\ep}\|_{L^\infty} + c^{m-1}+\|u_{0,\ep}\|_{L^\infty}^{m-1} \big)^{-1}\| \nab u_{0,\ep} \|_{L^\infty}^{-1}.
\end{equation*}  
\end{remark}

\begin{remark}
    Together with the weak-strong uniqueness obtained in the previous subsection, this result shows that there exists a unique Lipschitz solution to \eqref{eq:PDE} on $[0,T_*)$.
\end{remark}

\section{Properties of entropy solutions}

Since uniqueness is an open question in general, we provide several important properties that are shared among the class of entropy solutions. Of course, these properties are in particular satisfied by the solution we have constructed above. 

\subsection{Dissipation estimates}

\begin{prop}
    Let $u$ be an entropy solution to \eqref{eq:PDE}. Let  $\Phi$ such that  
    \begin{equation*}
    \begin{cases}
         \displaystyle \frac{d}{dt} \Phi_\be(t) = \Phi_\be(t)^m(\bar u_0 - \Phi_\be(t)), \\
         \Phi_\be(t=0) = \beta.
    \end{cases}
    \end{equation*} 
    Then, 
    \begin{enumerate}[label=\roman*.]
        \item $t\mapsto \text{ess sup }_{\T^d }\  u(t)$ and $t\mapsto \text{ess inf }_{\T^d }\  u(t)$ are respectively nonincreasing and nondecreasing,
        \item for a.e. $t>0$ and $x\in \T^d$, $\Phi_{\text{ess inf } u_0 }(t) \le u(t,x) \le \Phi_{\text{ess sup } u_0}(t)$,
        \item for all $1\le p\le \infty$, $t\mapsto \|u(t)\|_{L^p}$ is nonincreasing.
    \end{enumerate}
\end{prop}
\begin{remark}
    The lower bound shows in particular that any entropy solution is positive for $t>0$ when $m<1$. On the other hand, there exist weak (non entropic) solutions that do not satisfy this property (see \cite{CGV22FastRegularisation}).

    More precisely, when $m<1$, we have for a.e. $t>0$ and $\T^d$, 
    \begin{equation}\label{eq:lowerbarrier}
         u(t,x) \ge \begin{cases}
            \displaystyle \bar u_0 - 2^{-1}\bar u_0 e^{-2^{-m}\bar  u_0^m t}, & t\in (\tau_{1/2},\infty),\\
            \displaystyle 2^{-1/(1-m)}(\bar u_0 t)^{1/(1-m)} ,& t\in [0,\tau_{1/2}],
        \end{cases}
    \end{equation}
    where $\tau_{1/2}$ is defined in Lemma \ref{eq:ODE}. Furthermore, the short-time scaling $t^{1/(1-m)}$ is that of self-similar solutions to \eqref{eq:PDE}, as given in \cite{CGV22FastRegularisation} for the Euclidean setting.
\end{remark}
\begin{proof}
\textbf{Weak maximum principle.} Let us prove that, if $u_0\le c$ a.e. on $\T^d$ for some $c\ge 0$, then $u(t,x)\le c$ for a.e. $t>0$ and $x\in \T^d$. We have in the sense of distributions
\begin{align*}
    \p_t (u-c)_+ &\le \div((u^m-c^m)_+\nab\g\ast u) -c^m\indic_{u>c } (u-\bar u).
\end{align*}
Thus, in the weak sense on $\R_+$:
\begin{align*}
    \frac{d}{dt}\int_{\T^d}(u(t)-c)_+dx \le -c^m \int_{\T^d} (u(t)-c)_+dx -c^m(c-\bar u_0) |\{u(t)>c\}|.
\end{align*}
Take $c\ge 0$ such that $u_0 \le c$ a.e. on $\T^d$, which implies that
\begin{equation*}
    \int_{\T^d}(u(t)-c)_+dx  \le \int_{\T^d} (u_0-c)_+dx =0.
\end{equation*}
We thus have $u(t,x) \le c$ for a.e. $t>0$ and $x\in \T^d$. The same argument gives that, if $u_0\ge c$ a.e. on $\T^d$ for some $c\ge 0$, then $u\ge c$ a.e. on $\R_+\times \T^d$. This implies that $t\mapsto \text{ess sup }_{\T^d }\  u(t)$ and $t\mapsto \text{ess inf }_{\T^d }\  u(t)$ are respectively nonincreasing and nondecreasing.

\bigbreak
\textbf{Estimate on the maximum/minimum.}
    Consider a differentiable function of time $c:\R_+\to \R$. Since $u$ is an entropy solution to \eqref{eq:PDE}, we have in the sense of distributions 
    \begin{align*}
        \p_t (u-c)_+ &\le \div((u^m-c^m)_+\nab\g\ast u)  -c^m\indic_{u>c } (u-\bar u) -\dot c \indic_{u>c}.
    \end{align*}
    Integrating in space gives
    \begin{align*}
        \frac{d}{dt}\int_{\T^d} (u-c)_+dx &\le -c^m \int_{\T^d} (u-c)_+ dx - (c^m (c-\bar u) + \dot c) |\{u>c\}|.
    \end{align*}
    As long as we take $c\equiv c(t)$ such that
    \begin{equation*}
        \dot c + c^m(c-\bar u) \ge 0,
    \end{equation*}
    we obtain $\int_{\T^d} (u-c)_+(t,x)dx \le \int_{\T^d} (u_0-c_0)_+ dx$. Take $c_0>0$ such that $u_0 \le c_0$ a.e. on $\T^d$ and we obtain that 
    \begin{equation*}
        u(t,x)\le \Phi_{\|u_0\|_{L^\infty}}(t),
    \end{equation*}
    for a.e. $t>0$ and $x\in \T^d$. In the same spirit, as long as
    As long as 
    \begin{equation*}
        \dot c + c^m(c-\bar u_0) \le 0,
    \end{equation*}
    we obtain $\int_{\T^d} (c-u)_+(t,x)dx \le \int_{\T^d} (c_0-u_0)_+ dx$. This implies
    \begin{equation*}
        u(t,x)\ge \Phi_{\text{ess inf }u_0}(t),
    \end{equation*}
    for a.e. $t>0$ and $x\in \T^d$.

    \bigbreak
    \textbf{Decreasing of $L^p$ norms.}
    We have for any $p > 1$,
    \begin{equation*}
        \p_t u^p \le \frac{mp}{p+m-1}\div (u^{p+m-1}\nab\g\ast u) -\frac{p(p-1)}{p+m-1} u^{m+p-1}(u-\bar u_0).
    \end{equation*}
    Using the same argument based on Jensen's inequality as in \eqref{ineq:Jensen}, we obtain that the $L^p$ norms are nonincreasing for any $1\le p\le \infty$.

\end{proof}

\subsection{Asymptotic behaviour}

It is expected that the solution converges towards its spatial average on the torus. We prove that the rate of this convergence is exponential in time when $m<1$ for the energy metric $\dot H^{-1}$ and any $L^p$ topologies, $1\le p\le \infty$. When $m\ge 1$, we assume that the initial condition is bounded from below to obtain such an exponential rate of convergence in the energy and $L^\infty$ metrics.

\begin{proof}[Proof of Theorem~\ref{thm:asymptotic_behaviour}]
    Let us first consider the case $m<1$. We recall that for all $t$, $\bar{u}(t) = \bar{u}_{0}$. Notice that 
    \begin{align*}
        \| u(t) - \bar{u}_{0}\|_{L^\infty} &\le  \|u(t)\|_{L^\infty}- \text{ess inf } u(t) \\
        &\le \Phi_{\text{ess sup } u_0}(t) - \Phi_{\text{ess inf } u_{0}}(t).
    \end{align*}
    Using Lemma \ref{eq:ODE} gives, for $t\ge \tau_{1/2}$,
    \begin{equation*}
        \| u(t) - \bar{u}_{0}\|_{L^\infty}\le (\|u_{0}\|_{L^\infty} - \bar{u}_{0}) e^{-\bar{u}_{0}^m t} + \hal \bar ue^{-2^{-m}\bar{u}_{0}^m t}.
    \end{equation*}
    Since $\g$ has mean zero, we have
    \begin{align*}
        \frac{d}{dt}\hal\int_{\T^d} \g\ast (u(t)-\bar{u}_{0}) (u(t)-\bar{u}_{0}) dx &=-\int_{\T^d}|\nab\g\ast (u(t)-\bar{u}_{0}) |^2u^m(t) dx \\
        &\le -\Phi_{\min u_{0}}(t)^m\int_{\T^d}|\nab\g\ast (u(t)-\bar{u}_{0})|^2dx.
    \end{align*}
    Integrating by parts and reminding that $\int_{\T^d}\g\ast ffdx \sim \|f\|_{\dot H^{-1}}^2$, one obtains
    \begin{equation*}
        \| u(t)-\bar{u}_{0}\|_{\dot H^{-1}}\le \| u_{0}-\bar{u}_{0}\|_{\dot H^{-1}} e^{-\int_0^t \Phi_{\min u_{0}}^m(\tau)d\tau},
    \end{equation*}
    and there is $C_m >0$ such that
    \begin{align*}
        \int_0^t \Phi_{\min u_{0}}^m(\tau)d\tau \ge C_m \bar{u}_{0}^m t.
    \end{align*}
    We now suppose $m>1$ and $u_{0}\ge c>0$. Therefore, one obtains as before
    \begin{align*}
        \| u(t)-\bar{u}_{0}\|_{L^\infty} &\le \Phi_{\max u_{0}}(t) - \Phi_{c}(t) \\
        &\le (\|u_{0}\|_{L^\infty} - \bar{u}_{0}) e^{-\bar{u}_{0}^m t}  + (\bar{u}_{0}- c) e^{-c^mt}.
    \end{align*}
    Moreover,
    \begin{align*}
        \frac{d}{dt}\hal\int_{\T^d} \g\ast (u(t)-\bar{u}_{0}) (u(t)-\bar{u}_{0}) dx &= -\int_{\T^d}|\nab\g\ast (u(t)-\bar{u}_{0})|^2 u^m(t) dx  \\
        &\le -c^m\int_{\T^d}|\nab\g\ast (u(t)-\bar{u}_{0}) |^2dx.
    \end{align*}
    We obtain
    \begin{equation*}
        \|u(t) -\bar{u}_{0}\|_{\dot H^{-1}} \le \|u_{0} - \bar{u}_{0}\|_{\dot H^{-1}} e^{-c^mt}.
    \end{equation*}
Let us finally study the convergence in $L^p$ norms, $1\le p<\infty$. Since $u$ is an entropy solution, we have 
\begin{align*}
    \frac{d}{dt}\int_{\T^d} |u(t)-\bar u_0| dx &\le \int_{\T^d} (|u^m-\bar u_0^m|-\sgn (u-\bar u_0) u^m) (u-\bar u_0) dx\\
    &= -\bar u_0^m \int_{\T^d} |u(t)-\bar u_0| dx.
\end{align*}
This gives for a.e. $t>0$,
\begin{equation*}
    \|u(t)-\bar u_0\|_{L^1}\le \|u_0-\bar u_0\|_{L^1} e^{-\bar u_0^m t}.
\end{equation*}
By interpolation, we have for all $1\le p<\infty$ and a.e. $t>0$,
\begin{align*}
   \|u(t)-\bar u_0\|_{L^p}&\le \|u(t)-\bar u_0\|_{L^\infty}^{1-\frac{1}{p}} \|u(t)-\bar u_0\|_{L^1}^\frac{1}{p} \\
    &\le (\bar u_0+ (mt)^{-1/m})^{1-\frac{1}{p}}  e^{-\frac{1}{p}\bar u_0^m t }.
\end{align*}
\end{proof}

\section{Estimates for the size of the support}\label{sec:support}

In this section, we quantify the evolution of the size of the support of solutions. This has been previously obtained in \cite{CGV2022vortex, CGV22FastRegularisation} for the one-dimensional (or equivalently radial multidimensional) case, where one can use the classical link between scalar conservation laws and the Hamilton-Jacobi equations. 

In our multidimensional (non-radial) setting, we consider the antiderivative of the decreasing rearrangement of solutions to \eqref{eq:PDE}. This quantity is the equivalent of the Hamilton-Jacobi formulation for one-dimensional scalar conservation laws. Nevertheless, we shall see in the following that, contrary to the radial setting, this quantity might be a mere subsolution to the associated Hamilton-Jacobi equation.

We note that the estimate \eqref{eq:lowerbarrier} already shows that any entropy solution to \eqref{eq:PDE} instantaneously fills the whole domain $\T^d$ when $m<1$. Therefore, in this section, we restrict ourselves to the case $m\ge 1$.

Let us introduce some standard notations. For a measurable set $E\subset\R^d$, $|E|$ denotes its Lebesgue measure. For $f:\T^d\to\R$ measurable, set $\{f>s\}:=\{x\in\T^d:\ f(x)>s\}$. The decreasing rearrangement $f^*:[0,|\T^d|]\to\R$ is defined by
$$
 f^{*}(s):=
 \begin{cases}
 \inf\{\tau\in\R:\ |\{f>\tau\}|\le s\}, & 0\le s<|\T^d|,\\
 \operatorname*{ess\,inf}_{\T^d} f, & s=|\T^d|.
 \end{cases}
$$
For a space-time function $u\equiv u(t,x)$, we write $u_*(t,s) \equiv u(t)_*(s)$. Since we assume $|\T^d|=1$, the volume of the domain will be transparent in the sequel.

\subsection{Decreasing rearrangement and the Hamilton-Jacobi equation}\label{subsec:decreasing_rearrangement}

Let us formally consider a solution $u$ to \eqref{eq:PDE}. Define
\begin{equation*}
    k(t,s ) := \int_0^s u_*(t,\sigma ) d\sigma.
\end{equation*}
Formally, $k$ solves the following equation:
\begin{equation}\label{eq:primitive}
    \begin{cases}
    \p_t k + (\p_s k)^m_+ (k - s\bar u_0) = 0, & (t,s)\in (0,\infty)\times (0,1),\\
    k\vert_{s=0} = 0,\quad  
    k\vert_{s=1} = \bar u_0, \\
    k\vert_{t=0} = k_0.
\end{cases}    
\end{equation}
As we shall see in the following, it is not obvious that, given an entropy solution $u$ to \eqref{eq:PDE}, the integral quantity $k$ actually satisfies \eqref{eq:primitive}. In fact, when $m>1$, shocks appear in a finite time. This may lead to a dissipation mechanism such that $k$ is actually just a subsolution. This is a stricking difference with the radial case, where the integral of $u$ is a solution to the equation \eqref{eq:primitive}. 

The equation \eqref{eq:primitive} satisfying a comparison principle, we shall use the framework of viscosity solutions. 

\begin{mydef}\label{def:viscosity}
    Define the subdifferential (resp. superdifferential) of a function $f:(0,\infty)\times (0,1)\to \R$ by
    \begin{align*}
        D^{-}f(t,s) := \big\{ p\in \R^2 : \liminf_{(\tau,\sigma)\to (t,s)} \frac{f(\tau,\sigma)-f(t,s)-p_1(\tau -t) -p_2 (\sigma- s)}{|\tau -t| + |\sigma-s|} \ge 0\big\}, \\
        D^{+}f(t,s) := \big\{ p\in \R^2 : \limsup_{(\tau,\sigma)\to (t,s)} \frac{f(\tau,\sigma)-f(t,s)-p_1(\tau -t) -p_2 (\sigma- s)}{|\tau -t| + |\sigma-s|}  \le 0\big\} .
    \end{align*}
    We say that $k$ is a viscosity subsolution to \eqref{eq:primitive} if it is upper semicontinuous and, given any $t>0$, $s\in (0,1)$, and $(p_1,p_2)\in D^{+}k(t,s)$, we have 
    \begin{equation*}
        p_1 + (p_2)_+^m (k(t,s)-s\bar u_0) \le 0, \quad k(0,s)\le k_0(s), \quad k(t,0)\le 0,\quad k(t,1) \le \bar u_0.
    \end{equation*}
    We say that $k$ is a viscosity supersolution to \eqref{eq:primitive} if it is lower semicontinuous and, given any $t>0$, $s\in (0,1)$, and $(p_1,p_2)\in D^{-}k(t,s)$, we have 
    \begin{equation*}
        p_1 + (p_2)_+^m (k(t,s)-s\bar u_0) \ge 0, \quad k(0,s)\ge k_0(s), \quad k(t,0)\ge 0,\quad k(t,1) \ge \bar u_0.
    \end{equation*}
    We say that $k\in C^0((0,\infty)\times (0,1))$ is a viscosity solution to \eqref{eq:primitive} if it is both a viscosity supersolution and subsolution.
\end{mydef}

\begin{comment}
\begin{mydef}
    We say that $k$ is a viscosity subsolution to \eqref{eq:primitive} if it is upper semicontinuous and, given any $t>0$, $s\in (0,1)$, and $\varphi\in C^2$ such that $k-\varphi$ has a local maximum at $(t,s)$,
    \begin{equation*}
        \p_t\varphi(t,s) + (\p_s\varphi(t,s))_+^m (k(t,s)-s\bar u_0) \le 0, \quad k(0,s)\le k_0(s), \quad k(t,0)\le 0,\quad k(t,1) \le \bar u_0.
    \end{equation*}
    We say that $k$ is a viscosity supersolution to \eqref{eq:primitive} if it is lower semicontinuous and, given any $t>0$, $s\in (0,1)$, and $\varphi\in C^2$ such that $k-\varphi$ has a local minimum at $(t,s)$,
    \begin{equation*}
        \p_t\varphi(t,s) + (\p_s\varphi(t,s))_+^m (k(t,s)-s\bar u_0) \ge 0, \quad k(0,s)\ge k_0(s), \quad k(t,0)\ge 0,\quad k(t,1) \ge \bar u_0.
    \end{equation*}
    We say that $k\in C^0((0,\infty)\times (0,1))$ is a viscosity solution to \eqref{eq:primitive} if it is both a viscosity supersolution and subsolution.
\end{mydef}
\end{comment}
It is a rather classical result that \eqref{eq:primitive} has a unique viscosity solution, and that viscosity solutions satisfy a comparison principle (this is a straightforward adaptation from \cite[Theorems 4.3, 4.4]{CGV2022vortex}). 

Given the entropy solution $u$ constructed in Section 1, we recall the definition
\begin{equation*}
    k(t,s) := \int_0^s u_*(t,\sigma)d\sigma.
\end{equation*}
Our aim is to show that $k$ satisfies a comparison principle with any viscosity supersolution (see Corollary \ref{prop:comparisonprinciple}). Before this, we state the following lemma.
\begin{lemma}\label{lem:subsolution}
    Let $u_\ep$ be the compact sequence of section \ref{subsec:compactness}, and $u$ be the entropy solution constructed from it. Define the quantities
    \begin{equation*}
        k_\ep(t,s) := \int_0^s (u_\ep)_*(t,\sigma)d\sigma, \quad k(t,s) := \int_0^s u_*(t,\sigma)d\sigma.
    \end{equation*}
    Then, $k_\ep$ is a classical subsolution to \eqref{eq:primitive}. Moreover, $k_\ep \to k$ pointwise on $\R_+\times (0,1)$ and satisfies the one-sided Lipschitz estimate
    \begin{equation*}
        \p_t k_\ep \le C(\| u_0\|_{L^\infty}).
    \end{equation*}
\end{lemma}

\begin{comment}

\begin{remark}
    It is an open question whether $k$ is actually a true solution to the equation. When $m=1$, this can be achieved using mollification arguments. When $m\neq 1$, considering the equation set on the Euclidean space $\R^d$, we can show that the function
    \begin{equation*}
        u(t,x) := S(t)^{-1}\indic_{\om_d|x|^d< S(t) },\quad S(t) := (u_0^{-m} + mt)^{1/m}
    \end{equation*}
    is an entropy solution when $m>1$ (see \cite{CGV2022vortex}). Its decreasing rearrangement is
    \begin{equation*}
        u_*(t,s) =  \begin{cases}
            S(t)^{-1} & s< S(t), \\
            0    & s\ge  S(t).
        \end{cases}
    \end{equation*}
    Using the same notation as before, we have
    \begin{equation*}
        k(t,s) = \begin{cases}
            s /S(t) & s<S(t),\\
           1 & s\ge S(t).
        \end{cases} 
    \end{equation*}
    Since $\p_s k(t,s) =  u_*(t,s) = S(t)^{-1}\indic_{s<S(t)}$, we have
    \begin{align*}
         \p_t k(t,s) &= -\frac{s}{S(t)^2}\dot S(t) \indic_{s<S(t)} + \dot S(t)\d_{s=S(t)} - \dot S(t) \d_{s=S(t)} \\
         &= -\frac{s}{S(t)^{m+1}}\indic_{s<S(t)} \\
         &= -(\p_s k)^m(t,s)k(t,s).
    \end{align*}
    This shows that one can expect, for entropy solutions satisfying the Rankine-Hugoniot condition, that $k$ be a true solution to \eqref{eq:primitive} (and not a mere subsolution).

\end{remark}

\end{comment}

\begin{proof}
    We first notice that $\p_s k_\ep =( u_\ep)_* \in L^\infty$. Defining $\mu_\ep(\theta) := |\{(u_\ep)_* >\theta \}|$, we have from \cite[Corollary 9.2.1]{Rakotoson2008}
    \begin{align*}
        \int_0^{\mu_\ep(\theta)}\p_t (u_\ep)_*(t,\sigma)d\sigma &= \int_{\{u_\ep >\theta \}} \p_t u_\ep(t,x)dx \\
        &= \int_{\{u_\ep>\theta \}}\div(u_\ep^m\nab\g\ast u_\ep)dx + \ep\int_{\{u_\ep>\theta \}} \D u_\ep dx.
    \end{align*}
    The viscous term is negative and can be discarded, since $\indic_{u_\ep>\theta }\D u_\ep = \D (u_\ep-\theta)_+ - |\nab u_\ep|^2 \d_{u_\ep = \theta}$. We thus obtain
    \begin{align*}
        \int_0^{\mu_\ep(\theta)}\p_t (u_\ep)_*(t,\sigma)d\sigma & \le \int_{\{u_\ep>\theta \}} \nab u_\ep^m\cdot\nab \g\ast u_\ep dx - \int_{\{u_\ep>\theta \}} u_\ep^m (u_\ep-\bar u_\ep)dx \\
        &= \int_{\T^d} \nab (u_\ep^m-\theta^m)_+\cdot\nab\g\ast u_\ep dx - \int_{\{u_\ep>\theta \}} u_\ep^m (u_\ep-\bar u_\ep)dx \\
        &= -\theta^m \int_{\{u_\ep>\theta\}} (u_\ep-\bar u_\ep)dx.
    \end{align*}
    We then send $\theta\to (u_\ep)_* (t,s)$ for some fixed $s\in (0,1)$, which gives
    \begin{equation*}
        \p_t \int_0^s (u_\ep)_* d\sigma = \int_0^s \p_t(u_\ep)_* d\sigma \le -( u_\ep)_*^m(t,s) \bigg(\int_0^s (u_\ep)_* d\sigma - s \bar u_\ep \bigg).
    \end{equation*}
    This shows that
    \begin{equation*}
        \p_t k_\ep + (\p_s k_\ep)^m(k_\ep - s\bar u_\ep) \le 0.
    \end{equation*}
    In particular, 
    \begin{equation*}
        \p_t k_\ep(t,s) \le C(\|u_0\|_{L^\infty}).
    \end{equation*}
    We know from \cite[Theorem 1.3.1]{Rakotoson2008} that the decreasing rearrangement is a contraction operator on any $L^p$, $1\le p\le \infty$. Since $u_\ep(t) \to u(t)$ strongly in $L^p$, $1\le p <\infty$ for a.e. $t>0$, we deduce that, for all $t>0$ and $s\in (0,1)$,
    \begin{align*}
       & |k_\ep(t,s)-k(t,s) |\le \| u_\ep(t)-u(t)\|_{L^p} s^{1/p_*} \le \| u_\ep(t)-u(t)\|_{L^p} \xrightarrow[\ep \to 0]{} 0, \\
       & \|\p_s k_\ep(t)-\p_s k(t)\|_{L^p}\le \| u_\ep(t)-u(t)\|_{L^p} \xrightarrow[\to\to 0]{}0.
    \end{align*}
    
\end{proof}

\begin{cor}[Comparison principle]\label{prop:comparisonprinciple}
    Let $\tilde k$ be a viscosity supersolution to \eqref{eq:primitive} and $k$ defined in Lemma \ref{lem:subsolution}. We have $\tilde k\ge k$ on $\R_+\times (0,1)$.
\end{cor}
\begin{proof}
    Consider the approximation scheme used to obtain $k$. We have already shown that $k_\ep$ is a classical subsolution to \eqref{eq:primitive}. Using the comparison principle for viscosity solutions (which is a straightforward adaptation from \cite[Theorem 4.3]{CGV2022vortex}), we have $\tilde k\ge k_\ep$ on $\R_+\times (0,1)$. Sending $\ep\to 0$ gives the result.
\end{proof}

\iffalse

\begin{figure}[H]
    \centering
    % Top row
    \begin{minipage}{0.45\textwidth}
        \centering
        \includegraphics[width=\linewidth]{Pictures/m=2/t=0.00.png}
    \end{minipage}
    \begin{minipage}{0.45\textwidth}
        \centering
        \includegraphics[width=\linewidth]{Pictures/m=2/t=2.50.png}
    \end{minipage}

    % Space between rows
    \vspace{0.3cm}

    % Bottom row
    \begin{minipage}{0.45\textwidth}
        \centering
        \includegraphics[width=\linewidth]{Pictures/m=2/t=5.00.png}
    \end{minipage}
    \begin{minipage}{0.45\textwidth}
        \centering
        \includegraphics[width=\linewidth]{Pictures/m=2/t=10.00.png}
    \end{minipage}

    \caption{Evolution of $u$ and $u_*$ for $m=2$ at different times.} 
\end{figure}

\begin{figure}[H]
    \centering
    % Top row
    \begin{minipage}{0.45\textwidth}
        \centering
        \includegraphics[width=\linewidth]{Pictures/m=0.5/t=0.00.png}
    \end{minipage}
    \begin{minipage}{0.45\textwidth}
        \centering
        \includegraphics[width=\linewidth]{Pictures/m=0.5/t=2.50.png}
    \end{minipage}

    % Space between rows
    \vspace{0.3cm}

    % Bottom row
    \begin{minipage}{0.45\textwidth}
        \centering
        \includegraphics[width=\linewidth]{Pictures/m=0.5/t=5.00.png}
    \end{minipage}
    \begin{minipage}{0.45\textwidth}
        \centering
        \includegraphics[width=\linewidth]{Pictures/m=0.5/t=10.00.png}
    \end{minipage}

    \caption{Evolution of $u$ and $u_*$ for $m=0.5$ at different times.} 
\end{figure}

\fi

\subsection{A well-chosen supersolution}

In order to prove our result on the size of the support of solutions to \eqref{eq:PDE}, we will need to compare them with a well-chosen supersolution. The main difference from the Euclidean setting of \cite{CGV2022vortex} is that, in our periodic setting, it seems difficult to use analytical formulas for viscosity solutions of \eqref{eq:primitive} when a rarefaction wave occurs. We thus exhibit a supersolution in Lemma \ref{lemma:supersolution}. We mention that some rarefaction waves may not appear in our periodic setting, a sharp difference from the Euclidean setting of \cite{CGV2022vortex}. Indeed, the solution constructed in Proposition \ref{prop:twovortex} of the appendix is a viscosity solution, whereas there must be a rarefaction wave in the nonperiodic case (see \cite[Section 4.4]{CGV2022vortex}).

\paragraph{The Rankine-Hugoniot condition.}

Using the continuity of the mass at a shock situated at $S\in [0,1]$,
$$k(t,S_-) = k(t,S_+),$$
we obtain the following Rankine-Hugoniot condition:
\begin{equation*}
    \frac{dS}{dt} = (k(t,S(t)) - \bar u_0 S(t)) \frac{\p_sk(t,S_+)^m- \p_sk(t,S_-)^m}{\p_sk(t,S_+)-\p_sk(t,S_-)}.
\end{equation*}
This condition allows us to obtain some specific solutions to \eqref{eq:primitive}.

We find a suitable viscosity supersolution of~\eqref{eq:primitive} that has an instantaneous growth for the support. Using the comparison principle from Corollary \ref{prop:comparisonprinciple}, we will be able to prove that the same property holds for solutions of \eqref{eq:PDE} constructed by vanishing viscosity. 

\begin{lemma}\label{lemma:supersolution}
    Let $C>0$ small enough, such that $C\bar u_0 \le 1$. Consider the map $ \sigma(u) := u^\frac{m}{m-1}$, and $\al \in (0,1)$ large enough, such that $2(1-\al)\sigma'(1)\le 1$. Let $(S_1,S_2,S_3)$ be respectively defined as
    \begin{align*}
        &S_1 := C\al \bar u_0,\\ 
        &\dot S_2 (t) = -(1-\al)^{m-1}\bar u_0^m \frac{S_3-S_2 + (1-\al)\sigma'(1) }{(S_3-S_2)^{m-1}}\sigma'(1)^{m-1},  \\
        &\dot S_3 (t) = (1-\al)^{m-1}\bar u_0^m \frac{1-S_3}{(S_3-S_2)^{m-1}}\sigma'(1)^{m-1},
    \end{align*}
    with initial condition $1\ge S_3^0> S_2^0> S_1$. Let us introduce the notation $\displaystyle u := \frac{s-S_2}{S_3-S_2}$. The function
    \begin{equation*}
     \tilde k_{C,S_2^0,S_3^0}(t,s) := \begin{cases}
        C^{-1}s & 0\le s \le S_1, \\
        \al \bar u_0 & S_1 < s \le S_2, \\
    \sigma(u) (1-\al )\bar u_0 + \al \bar u_0 & S_2 < s \le S_3, \\
    \bar u_0  & S_3 < s \le 1,
    \end{cases}
    \end{equation*}
    is a viscosity supersolution to \eqref{eq:primitive} on $[0,T_*]$, where we introduce
    \begin{align*}
        T_* &:= \inf\{t>0 : S_2(t) = S_1\},\\
         T^* &:=\inf\{t>0 : 2S_3(t) = 1 +S_3^0\},
    \end{align*}
    satisfying
    \begin{align*}
        T_* \ge C_m\bigg(\frac{S_2^0-S_1}{\bar u_0}\bigg)^m, \qquad
        T^* \ge C_m\bigg(\frac{1-S_3^0}{\bar u_0}\bigg)^m .
    \end{align*}
    for some constant $C_m>0$ depending only on $m$. 
    We finally record the following:
    \begin{equation*}
    \begin{cases}
      \forall t\in [0,T_*], &\quad S_2(t) \ge  S_2^0 - C_m\bar u_0 t^\frac{1}{m} , \\
      \forall t\in [0,\min(T^*,T_*)],&\quad  S_3(t) \ge S_3^0 + C_m' (1-\al)^{\frac{m-1}{m}} \bar u_0 t^\frac{1}{m}  \\
     \forall t\ge 0, &\quad S_3(t) \le S_3^0 + C_m'' \bar u_0 t^\frac{1}{m}.
    \end{cases}
    \end{equation*}
    
\end{lemma}
\begin{remark}\label{cor:twovortexcomparison}
Letting aside the proof of this Lemma for a moment, we see that, sending $C\to 0$ and $S_1^0\to 0$, $S_2^0\to S_3^0\equiv s_0\in (0,1)$, one obtains a function $\tilde k$ that satisfies the same comparison principles as viscosity supersolutions to \eqref{eq:primitive}. Indeed, if the map $\tilde k_{C,S^0_2,S_3^0}$ satisfies $\tilde k_{C,S^0_2,S_3^0} \ge \varphi$, uniformly in the parameters $C, S_2^0$, for some bounded function $\varphi$. Then, $\tilde k \ge \varphi$. The function $\tilde k$ is lower semicontinuous and satisfies on $[0,T_*]$, $\tilde k(t,0)= 0$ and
\begin{equation*}
    \tilde k (t,s) := \begin{cases}
        \al \bar u_0 & 0< s \le S_2(t), \\
        \sigma(u) (1-\al )\bar u_0 + \al \bar u_0 & S_2(t) < s \le S_3(t), \\
    \bar u_0  & S_3(t) < s \le 1.
    \end{cases}
\end{equation*}
Moreover, $\tilde k$ satisfies 
\begin{equation}\label{eq:fronttrackings}
    \begin{cases}
      \forall t\in [0,T_*], &\quad S_2(t) \ge  s_0 - C_m\bar u_0 t^\frac{1}{m} , \\
      \forall t\in [0,\min(T^*,T_*)],&\quad  S_3(t) \ge s_0 + C_m' (1-\al)^{\frac{m-1}{m}} \bar u_0 t^\frac{1}{m}  \\
     \forall t\ge 0, &\quad S_3(t) \le s_0 + C_m'' \bar u_0 t^\frac{1}{m},
    \end{cases}
    \end{equation}
where $T_* \ge C_m (s_0/\bar u_0)^m$ and $T^* \ge C_m \big((1-s_0 )/ \bar u_0\big)^m$.
\end{remark}

\begin{proof}[Proof of Lemma \ref{lemma:supersolution}]
Some details are skipped in the proof. A more detailed proof showing that a function (in this case a single and a double-vortex) is a viscosity solution can be found in the appendix.
By construction, $ \tilde k_{C,S_2^0,S_3^0}$ is continuous on $[0,T_*]\times [0,1]$, and $C^1$ outside of $\Gamma_1 \cup \Gamma_3$, where $\Gamma_i := \{(t,s) \in [0,T_*]\times [0,1] : s= S_i \}$. Consider $(t,s)$ such that $s< S_1$. We have
\begin{align*}
    &\p_t \tilde k_{C,S_2^0,S_3^0} (t,s) + (\p_s \tilde k_{C,S_2^0,S_3^0})^m_+(t,s) (\tilde k_{C,S_2^0,S_3^0}(t,s)-s\bar u_0) \\
    &= C^{-m}(C^{-1}-\bar u_0) s \ge 0.
\end{align*}
If $(t,s)$ is such that $s>S_3$ or $S_1<s<S_2$, the equation is trivially satisfied. Let us then consider $(t,s)$ such that $S_2 < s < S_3$. First,
\begin{align*}
    \p_t \tilde k_{C,S_2^0,S_3^0}(t,s) &= \frac{(1-\al)\bar u_0}{S_3-S_2} \sigma'(u) \big( -(1-u)\dot S_2 - u \dot S_3  \big) \\
    &= \bigg(\frac{(1-\al)\bar u_0}{S_3-S_2}\bigg)^m \bar u_0 \sigma'(1)^{m-1} \sigma'(u)  \big( (1-u) (S_3-S_2 + (1-\al)\sigma'(1)) - u(1-S_3)\big).
\end{align*}
Moreover,
\begin{align*}
    &\p_s \tilde k_{C,S_2^0,S_3^0}(t,s)^m(\tilde k_{C,S_2^0,S_3^0}(t,s) -s \bar u_0) \\
    &= \bigg(\frac{(1-\al)\bar u_0}{S_3-S_2}\bigg)^m \bar u_0 \sigma'(u)^m \big( (1-\al) \sigma(u) + \al - s \big) .
\end{align*}
Therefore, 
\begin{align*}
    &\p_t \tilde k_{C,S_2^0,S_3^0}(t,s) +\p_s \tilde k_{C,S_2^0,S_3^0}(t,s)^m(\tilde k_{C,S_2^0,S_3^0}(t,s) -s \bar u_0) \\
    &= \bigg(\frac{(1-\al)\bar u_0}{S_3-S_2}\bigg)^m \bar u_0\sigma'(u)\bigg( \sigma'(1)^{m-1}   \big( (1-u) (S_3-S_2 + (1-\al)\sigma'(1)) - u(1-S_3)\big) \\
    &+ \sigma'(u)^{m-1} \big( (1-\al) \sigma(u) + \al - s \big)  \bigg).
\end{align*}
The right-hand side above goes to zero as $u\to 1$ (hence $s\to S_3$). Moreover, the function 
\begin{multline*}
    u\mapsto \sigma'(1)^{m-1}   \big( (1-u) (S_3-S_2 + (1-\al)\sigma'(1)) - u(1-S_3)\big) \\
    + \sigma'(u)^{m-1} \big( (1-\al) \sigma(u) + \al - s \big)
\end{multline*}
is differentiable in $u$, and using  $\sigma'(u)^{m-1} = \sigma'(1)^{m-1} u$, $s= (S_3 - S_2) u + S_2$ its derivative is 
\begin{align*}
    &\sigma'(1)^{m-1}   \big( - (S_3-S_2 + (1-\al)\sigma'(1)) - (1-S_3)\big) 
    + \sigma'(1)^{m-1} \big( (1-\al) \sigma(u) + \al - s \big) \\
    &+ \sigma'(1)^{m-1}u \big( (1-\al) \sigma'(u) - S_3+S_2 \big),
\end{align*}
whose sign is that of
\begin{align*}
     &  - (S_3-S_2 + (1-\al)\sigma'(1)) - (1-S_3) 
    +  (1-\al) \sigma(u) + \al - s  \\
    &+ u \big( (1-\al) \sigma'(u) - S_3+S_2 \big)\\
    &\le  - (S_3-S_2 + (1-\al)\sigma'(1)) +S_3- s  
    + u \big( (1-\al) \sigma'(u) - S_3+S_2 \big) \\
    &\le  - (S_3-S_2 + (1-\al)\sigma'(1)) + S_3-S_2 + (1-\al)\sigma'(1) \\
    &= 0.
\end{align*}

Therefore, 
\begin{equation*}
    \p_t \tilde k_{C,S_2^0,S_3^0} + (\p_s \tilde k_{C,S_2^0,S_3^0})^m_+ (\tilde k_{C,S_2^0,S_3^0}-s\bar u_0) \ge 0,
\end{equation*}
for all $s\in (S_2, S_3)$.

We finally notice that
\begin{align*}
     D^{-}\tilde k_{C,S_2^0,S_3^0}(t,S_1) = \emptyset, 
    & \quad D^{-}\tilde k_{C,S_2^0,S_3^0}(t,S_3(t)) = \emptyset.
\end{align*}
Therefore, $\tilde k_{C,S_2^0,S_3^0}$ is a viscosity supersolution on $[0,T_*]$. Let us now derive a lower bound for $T_*$. 
First, we have $S_2(t)\le 1$ for all $t\ge 0$, so that
\begin{align*}
    \frac{d}{dt}(S_3-S_2) &= (1-\al)^{m-1}\bar u_0^m  \sigma'(1)^{m-1} \frac{1-S_2 + (1-\al)\sigma'(1)}{(S_3-S_2)^{m-1}} \\
    &\ge (1-\al)^{m}\bar u_0^m \frac{1 }{(S_3-S_2)^{m-1}}\sigma'(1)^{m}.
\end{align*}
Solving this equation, one obtains
\begin{equation}\label{ecartS3S2lowerbound}
    \forall t\ge 0, \quad S_3(t)-S_2(t) \ge C_m(1-\al) \bar u_0 t^\frac{1}{m}.
\end{equation}
We also record the following: if $2(1-\al)\sigma'(1)\le 1$ and $t\le T_*$ so that $S_2\ge S_1\ge0$,
\begin{align*}
    \frac{d}{dt}(S_3-S_2) 
    &\le C(1-\al)^{m-1}\bar u_0^m  \frac{1}{(S_3-S_2)^{m-1}}.
\end{align*}
Therefore,
\begin{equation}\label{ecartS3S2upperbound}
   \forall t\in [0,T_*],\quad  S_3(t)-S_2(t)\le  C_m'(1-\al)^{\frac{m-1}{m}}\bar u_0 t^\frac{1}{m}.
\end{equation}
For $t\in [0,T_*]$, we have $S_3-S_2 \le 1$. Using also \eqref{ecartS3S2lowerbound} and $2(1-\al)\sigma'(1)\le 1$,
\begin{align*}
    \dot S_2 &= -(1-\al)^{m-1}\bar u_0^m \frac{S_3-S_2 + (1-\al)\sigma'(1) }{(S_3-S_2)^{m-1}}\sigma'(1)^{m-1} \\
    &\ge -C_m \bar u_0 t^{-\frac{m-1}{m}}.
\end{align*}
Therefore,
\begin{equation*}
   \forall t\in [0,T_*],\quad  S_2(t) \ge S_2^0 -C_m \bar u_0 t^\frac{1}{m}.
\end{equation*}
This gives the lower bound on $T_*$. Finally, define for all $S_3^0 \le \lambda \le 1$
\begin{equation*}
    T_\lambda^* := \inf\{t>0 : S_3(t) = \lambda\}.
\end{equation*}
For all $t< \min(T_\lambda^*, T_*)$, we have using \eqref{ecartS3S2upperbound}
\begin{align*}
    \dot S_3(t) &\ge (1-\al)^{m-1}\bar u_0^m \frac{1-\lambda }{(S_3-S_2)^{m-1}}\sigma'(1)^{m-1}\\
    &\ge C_m (1-\la )(1-\al)^\frac{m-1}{m}\bar u_0 t^{-\frac{m-1}{m}}.
\end{align*}
Solving this equation gives, for all $t\le \min(T_\la^*, T_*)$,
\begin{equation*}
    S_3(t) \ge S_3^0 + C_m (1-\la )(1-\al)^\frac{m-1}{m}\bar u_0t^\frac{1}{m}.
\end{equation*}
Finally, using \eqref{ecartS3S2lowerbound},
\begin{align*}
    \dot S_3(t)& \le (1-\al)^{m-1}\bar u_0^m \frac{1 }{(S_3-S_2)^{m-1}}\sigma'(1)^{m-1} \\
    &\le C \bar u_0 t^{-\frac{m-1}{m}}.
\end{align*}
Thus, 
\begin{equation*}
   \forall t\ge 0,\quad  S_3(t) \le S_3^0 + C_m \bar u_0 t^\frac{1}{m}.
\end{equation*}
This gives a lower bound on $T_\la^*$. Taking $\displaystyle \la := \frac{1+S_3^0}{2}$ ends the proof.
\end{proof}

\subsection{Instantaneous growth of the support and waiting time}

In this subsection, we define
\begin{equation*}
    S(t) := \inf\{s\in (0,1) : k(t,s) = \bar u_0 \}.
\end{equation*}
We record the following lemma. 
\begin{lemma} 
\begin{equation*}
    S(t) = |\{s\in (0,1) : k(t,s) <\bar u_0\}| = |\supp u(t)|. 
\end{equation*}
\end{lemma}
\begin{proof}
For a.e. $t>0$, the quantity $k(t,\cdot)$ is continuous, increasing, and satisfies $k(t,1)=\int_{0}^{1}u_{*}(t,\sigma)d\sigma = \int_{\T^d}u(t)dx = \bar{u}_0$. Therefore, $S(t) = |\{s\in (0,1) : k(t,s) <\bar u_0\}|$. Moreover, $|\{u_*(t)>\theta\}| = |\{u(t) >\theta \}|$, for all $\theta \ge0$. In particular, $|\supp u(t) | = |\{u_*(t) >0\}| = |\{k(t)<\bar u_0\}|$. 
\end{proof}
Adapating the proof of~\cite{CGV2022vortex} we now prove Theorem~\ref{thm:waitingtime}.

\begin{proof}[Proof of Theorem~\ref{thm:waitingtime}]
We start by proving the first item, that is the instantaneous growth of the support. We consider a sequence $(d_i,s_i)_i$ such that
    \begin{equation*}
        d_i\to \limsup_{s\to S_0^{-}} \frac{\bar u_0 - k_0(s)}{(S_0-s)^{\frac{m}{m-1}}} = +\infty,
    \end{equation*}
    $s_i\to S_0$, and, for all $i\ge 0$,
    \begin{equation*}
        \bar u_0- k_0(s_i) \ge d_i(S_0-s_i)^\frac{m}{m-1}.
    \end{equation*}
    Consider the function $\tilde k_{C,S^0_2,S^0_3}$ constructed in Lemma \ref{lemma:supersolution}, with 
    \begin{equation*}
        \begin{cases}
            C := \ep\|\p_s k_0\|_{L^\infty}^{-1}, \\
            \al := k_0(s_i)/\bar u_0, \\
            S_3^0 := s_i, \\
            S_2^0 := s_i - \ep.
        \end{cases}
    \end{equation*}
    For any $\ep \in (0,1)$ and $s\in [0,1]$, $\tilde k_{C, S_2^0, S_3^0} (0,s) \ge k_0(s)$. By the comparison principle of Proposition \ref{prop:comparisonprinciple}, we obtain that $\tilde k_{C, S_2^0, S_3^0} (t,s) \ge k_0(t,s)$ for all $t\in [0,T_*]$ and $s\in [0,1]$. Sending $\ep\to 0$ as in Remark \ref{cor:twovortexcomparison}, one obtains $\tilde k \ge k$ on $[0,T_*]\times [0,1]$. Therefore, for all $t\in [0,\min(T^*,T_*)]$, $S(t) \ge S_3(t)$. We want to take $t_i < t $ and $\ep_i>0$ so that $S(t)\ge S_3(t) \ge S_3(t_i) \ge S_0 + \ep_i >S_0$, which suffices to demonstrate our result.
    \begin{align*}
        S(t)&\ge S_3(t) \\
        &\ge S_3(t_i) \\
        &\ge s_i + C_m d_i^\frac{m-1}{m}(S_0-s_i)\bar u_0^\frac{1}{m} t^\frac{1}{m}_i \\
        &\ge S_0 + \ep_i.
    \end{align*}
    It is enough to take $\ep_i := S_0 - s_i$ and $t_i := 2\big(C_m \bar u_0 d_i^\frac{m-1}{m}\big)^{-m}\to 0$.

This concludes the proof of the nonexistence of a waiting time. 

Consider now a solution $u^{m-1}\in C^{0,1}([0,T_*)\times \T^d)$ for some $T_*>0$. In particular we assume that $u_{0}^{m-1}$ is Lipschitz, so by property of the decreasing rearrangement, $(u_{0}^{m-1})_{*}$ is Lipschitz. Therefore $k_{0}$ satisfies
\begin{equation}\label{eq:conditionCarillo}
    \limsup_{s\to S_0^{-}}\frac{\bar u_0 - k_0(s)}{(S_0-s)^\frac{m}{m-1}} < +\infty.
\end{equation}
Moreover, since $u^{m-1}_* \in C^{0,1}([0,T_*)\times [0,1])$, we have that $u_*$ is itself Lipschitz inside its support. Finally, \eqref{eq:conditionCarillo} being propagated on $[0,T_*)$ implies that $k\in C^{0,1}([0,T_*)\times [0,1])$. Therefore, all the computations in order to obtain the equation on $k$ in Lemma~\ref{lem:subsolution} that were made at the level of the viscous approximation and sending $\varepsilon\to 0$ can now be made rigorously directly on the equation with $\eps=0$ and are now rigorous. In particular, we obtain the following equation on $k$ (rather than an inequality)
\begin{equation}
    \p_t k + (\p_sk)^m_+(k-s\bar u_0) = 0,
\end{equation}
a.e. on $[0,T_*)\times [0,1]$.  We consider the explicit ansatz from \cite[Equation (4.32)]{CGV2022vortex}, which we denote $\tilde k$. We have, for some $T^*>0$, that $\tilde k $ is a viscosity subsolution to 
\begin{equation}
    \p_t \tilde k + (\p_s \tilde k)_+^m \tilde k = 0,
\end{equation}
a.e. on $[0,T^*)\times \R_+$. Therefore, restricting this function in space on $[0,1]$, one obtains that $\tilde k$ is a subsolution to \eqref{eq:primitive}. The rest of the argument follows that of \cite[Corollary 4.9]{CGV2022vortex}: by the comparison principle, one concludes that $k\ge \tilde k$ on $[0,\min (T_*, T^*))\times [0,1]$. Therefore, for all $t\in [0,\min (T_*, T^*))$, we have $S(t)\le \tilde S(t) = S_0$, where $\tilde S(t) := \inf \{s\in (0,1) : \tilde k(t,s) = \bar u_0\}$. Finally, by continuity, $k(t,s) > s\bar u_0 $ in a neighborhood of $(0,S_0)$, so that $\p_t k\le 0$ in this same neighborhood. Therefore, $S$ is nondecreasing and $S(t) \ge S_0$.

This concludes the proof.  
\end{proof}

We show here numerical simulations suggesting the waiting time phenomenon. 

\begin{figure}[H]
    \centering
    % Top row
    \begin{minipage}{0.45\textwidth}
        \centering
        \includegraphics[width=\linewidth]{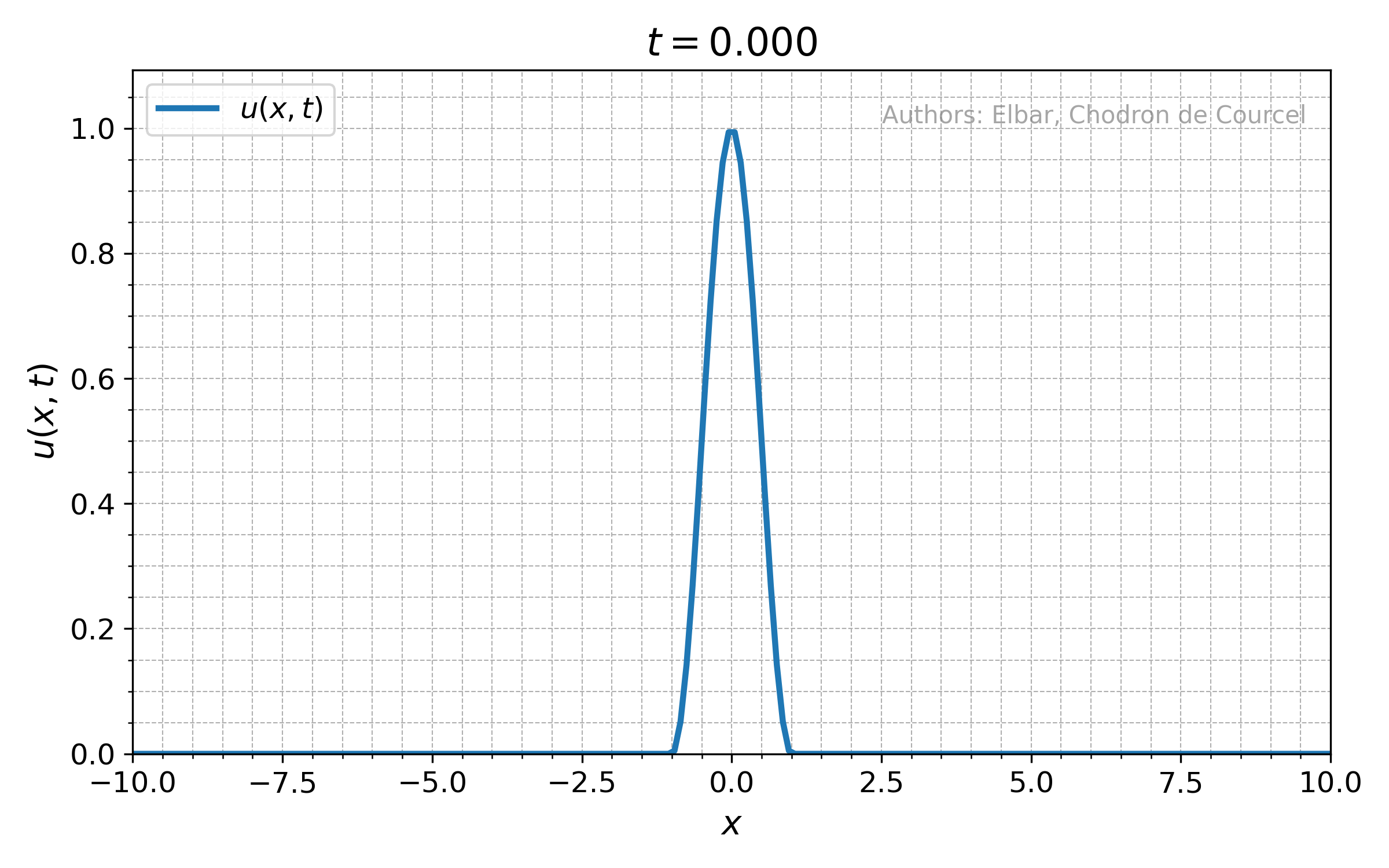}
    \end{minipage}
    \begin{minipage}{0.45\textwidth}
        \centering
        \includegraphics[width=\linewidth]{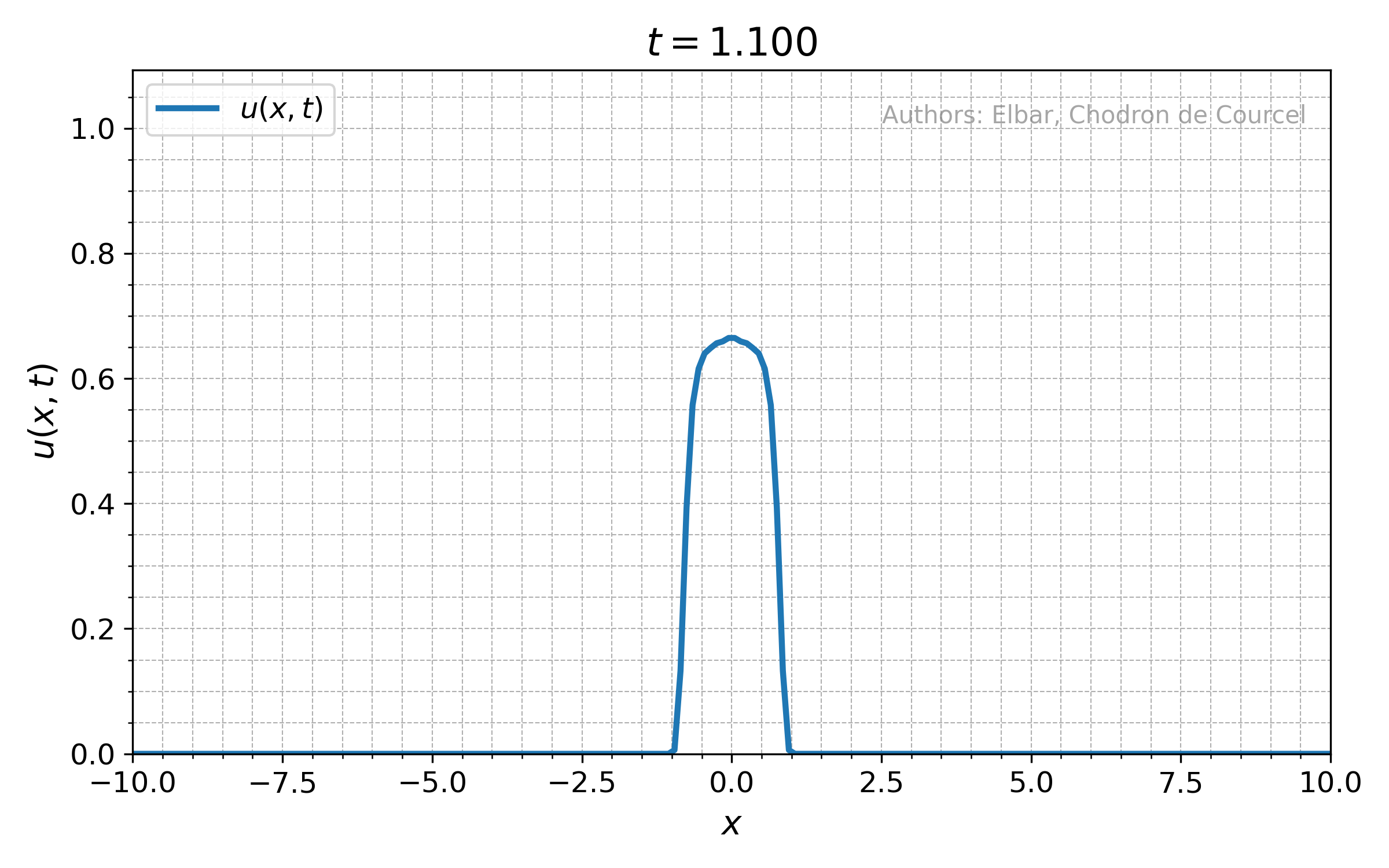}
    \end{minipage}

    % Space between rows
    \vspace{0.3cm}

    % Bottom row
    \begin{minipage}{0.45\textwidth}
        \centering
        \includegraphics[width=\linewidth]{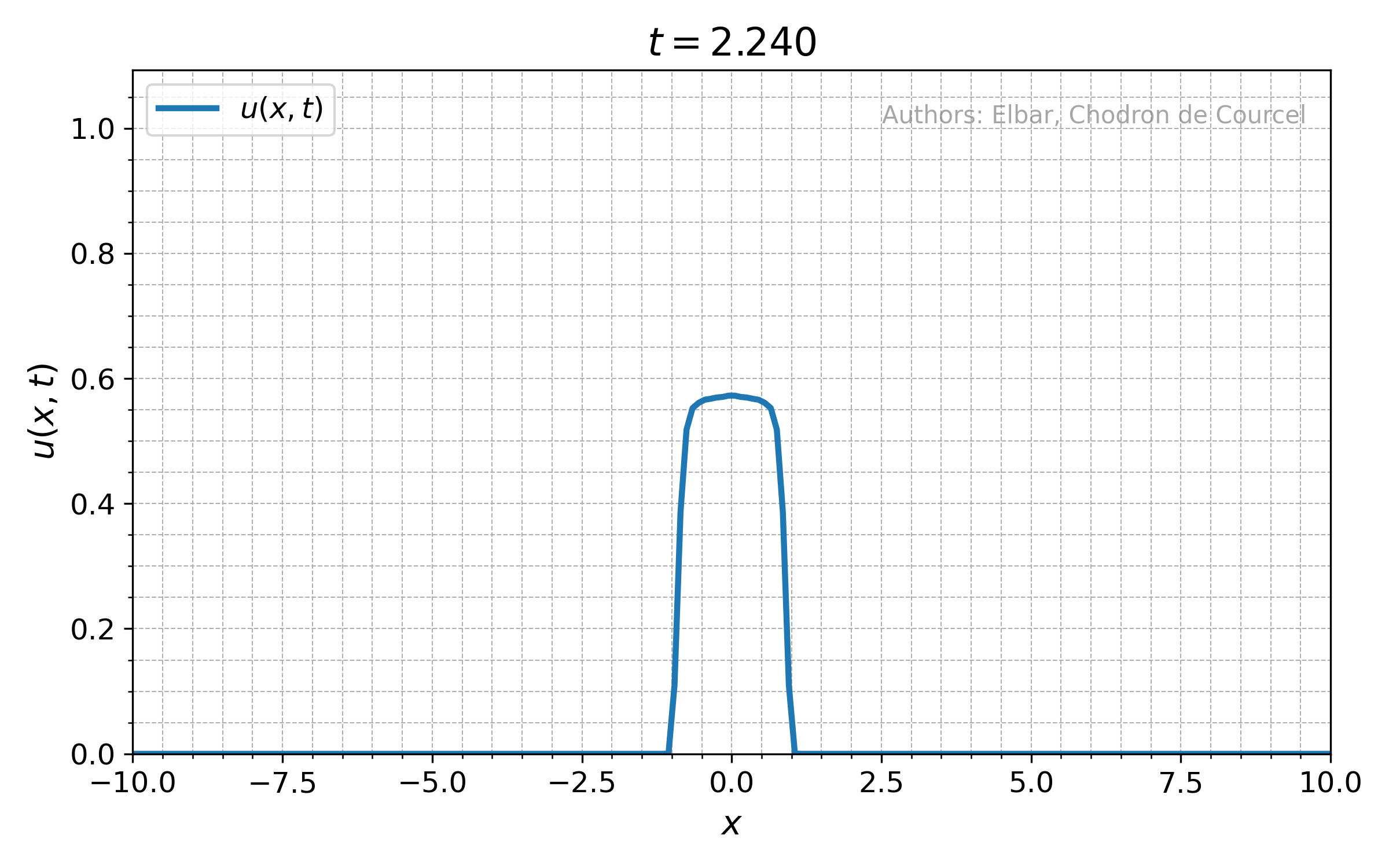}
    \end{minipage}
    \begin{minipage}{0.45\textwidth}
        \centering
        \includegraphics[width=\linewidth]{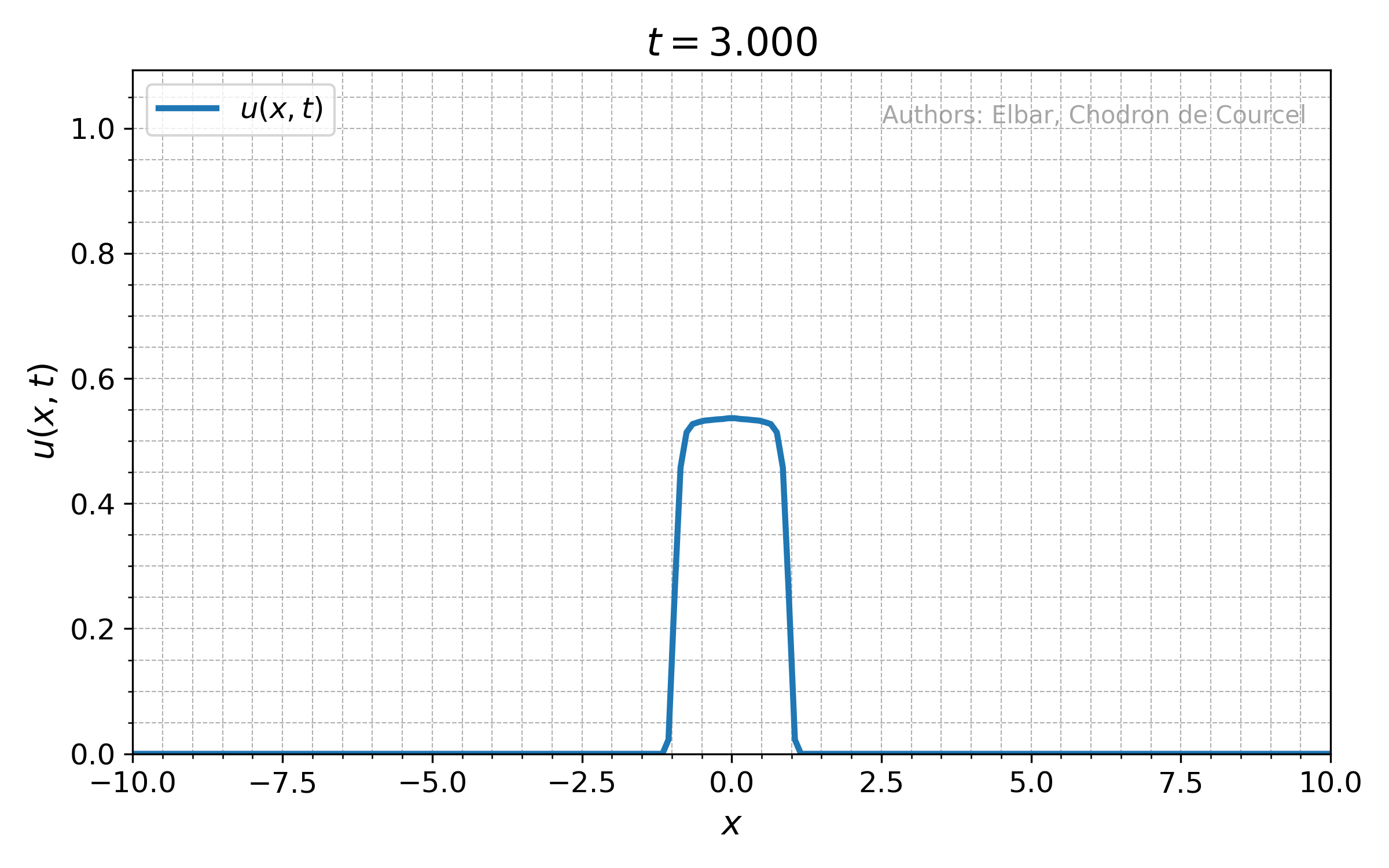}
    \end{minipage}

    \caption{Evolution of the system for $m=4$ at different times. Initially (at $t=0.000)$, the support is localized and remains the same up to $t=1.100$. After this point, and once enough pressure has built, the support starts growing.}
\end{figure}

\paragraph{Acknowledgements}
A.C.D.C. acknowledges support from the Fondation CFM, through the Jean-Pierre Aguilar fellowship.

C.E. was supported by the European Union via the ERC
AdG 101054420 EYAWKAJKOS project.

\appendix

\section{Some viscosity solutions}

We show here examples of viscosity solutions to~\eqref{eq:primitive}. We start with a simple single-vortex solution. 
\begin{prop}[Existence of a single-vortex viscosity solution]
Let $m\ge 1$. Consider the Equation~\eqref{eq:primitive}
with boundary conditions
\begin{equation}\label{eq:boundary_single}
  k(t,0)=0,\quad k(t,1)=\bar u_0 \quad\text{for } t\in(0,T),
\end{equation}
and initial datum $k(0,\cdot)=k_0(\cdot)$. Assume there exist $S_1^0,S_2^0$ with $0\le S_1^0<S_2^0\le 1$ such that
$$
k_0(s)=
\begin{cases}
  0, & s<S_1^0,\\
  \displaystyle \frac{\bar u_0}{S_2^0-S_1^0}\,(s-S_1^0), & S_1^0\le s<S_2^0,\\
  \bar u_0, & S_2^0\le s<1.
\end{cases}
$$
Let $S_1,S_2\in C^1 ([0,T) )$ be the unique solution of the Rankine-Hugoniot condition
\begin{equation}\label{eq:RH-single}
  \frac{dS_1}{dt} = -\bar u_0^{\,m}\frac{S_1}{ (S_2-S_1 )^{m-1}},
  \qquad
  \frac{dS_2}{dt} = \bar u_0^{\,m}\frac{1-S_2}{ (S_2-S_1 )^{m-1}},
\end{equation}
with initial data $S_1(0)=S_1^0$, $S_2(0)=S_2^0$. 
Define $k:(0,T)\times(0,1)\to\mathbb{R}$ by
$$
k(t,s)=
\begin{cases}
  0, & s<S_1(t),\\
  \displaystyle \alpha(t)\,(s-S_1(t)), & S_1(t)\le s<S_2(t),\\
  \bar u_0, & s\ge S_2(t),
\end{cases}
\qquad
\alpha(t):=\frac{\bar u_0}{S_2(t)-S_1(t)}.
$$
Then $k\in C^0 ([0,T)\times[0,1] )$, satisfies the boundary conditions \eqref{eq:boundary_single} and $k(0,\cdot)=k_0$, and is a viscosity solution of Equation~\eqref{eq:primitive}    on $(0,T)\times(0,1)$ in the sense of Definition~\ref{def:viscosity}.
\end{prop}

\begin{proof}
    First, $k$ is continuous and the boundary conditions are satisfied with equality.
    
    \underline{Step 1: Interior of the three parts.} On the left and right part, that is in $\{s<S_1(t)\}$ and $\{s>S_2(t)\}$, $k=0$ or $k=\bar{u}_0$ is constant. Then $D^{\pm}k(t,s)=\{(0,0)\}$ and we immediately check that $k$ is a viscosity solution at these points. In the middle part, $\{S_{1}(t)<s<S_{2}(t)\}$, we have $\partial_sk=\alpha$, $\partial_t k = \alpha'(s-S_1)-\alpha S_1'$. Then, using the Rankine-Hugoniot condition we find that $k$ is also a viscosity solution at these points.\\
    \underline{Step 2: Semi-differentials at the interface.} We define $\Gamma_{1}=\{(t,s): s=S_{1}(t)\}$ and $\Gamma_{2} = \{(t,s): s=S_{2}(t)\}$. From the definition of $k$ we make the observation that at $\Gamma_1$, the graph of $k$ is convex in $s$, and at $\Gamma_2$, the graph is concave.  On $\Gamma_1$, the gradients from the two sides are
$(0,0)$ (left part) and $(-\alpha S_1',\alpha)$ (middle part as $\partial_s k = \alpha$, $\partial_t k\to -\alpha S_1'$ as $s\downarrow S_1$). Hence since the semi-differential is the convexe envelope of the two points: 
$$
D^{-}k(t,S_1(t))=\{(-S_1' p_2,\,p_2): 0\le p_2\le \alpha\},\qquad
D^{+}k(t,S_1(t))=\emptyset.
$$

On $\Gamma_2$, the gradients are $(-\alpha S_2',\alpha)$ (middle) and $(0,0)$ (right). Thus
$$
D^{+}k(t,S_2(t))=\{(-S_2'p_2,\,p_2): 0\le p_2\le \alpha\},\qquad
D^{-}k(t,S_2(t))=\emptyset.
$$

\underline{Step 3: Viscosity solution} It only remains to prove that $k$ is a supersolution at $\Gamma_1$ and a subsolution at $\Gamma_2$.

Let $(p_1,p_2)\in D^{-}k(t,S_1(t))$, so $p_1=-S_1'p_2$, $0\le p_2\le \alpha$.  
Since $k(t,S_1)=0$, we have
$$
-S_1' p_2 + p_2^{\,m}\,(-S_1\bar u_0)
=p_2(-S_1'-\bar u_0 S_1\,p_2^{\,m-1}).
$$
Using $-S_1'=\bar u_0 S_1\alpha^{m-1}$ we obtain
$$
-S_1'-\bar u_0 S_1\,p_2^{\,m-1}
=\bar u_0 S_1(\alpha^{m-1}-p_2^{\,m-1})\ge0,
$$
and thus $k$ is a supersolution at $\Gamma_1$. 

Let $(p_1,p_2)\in D^{+}k(t,S_2(t))$, so $p_1=-S_2'p_2$, $0\le p_2\le\alpha$.  
Here $k(t,S_2)-S_2\bar u_0=\bar u_0(1-S_2)>0$, hence
$$
-S_2' p_2 + p_2^{\,m}\bar u_0(1-S_2)
=p_2\big(-S_2'+\bar u_0(1-S_2)p_2^{\,m-1}\big).
$$
Since $S_2'=\bar u_0(1-S_2)\alpha^{m-1}$, this equals
$$
p_2\,\bar u_0(1-S_2)\,(p_2^{\,m-1}-\alpha^{m-1})\le0.
$$
Thus $k$ is a subsolution at $\Gamma_2$.
    
\end{proof}

\paragraph{Nonexistence of rarefaction waves.}
We move to the existence of a two-vortex viscosity solution on the torus. We make the important remark that in the Euclidean setting, such a quantity will never define a viscosity solution, because a rarefaction wave appears near $S_3(t)^-$ (\cite[Section 4.4]{CGV2022vortex}). We provide the following numerical experiment that shows this behavior. 
\begin{figure}[H]
    \centering
    % Top row
    \begin{minipage}{0.45\textwidth}
        \centering
        \includegraphics[width=\linewidth]{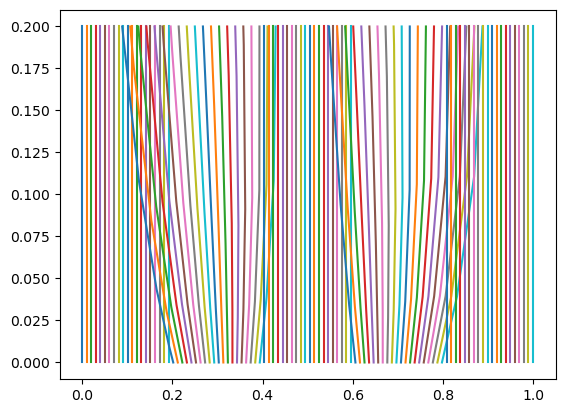}
    \end{minipage}
    \begin{minipage}{0.45\textwidth}
        \centering
        \includegraphics[width=\linewidth]{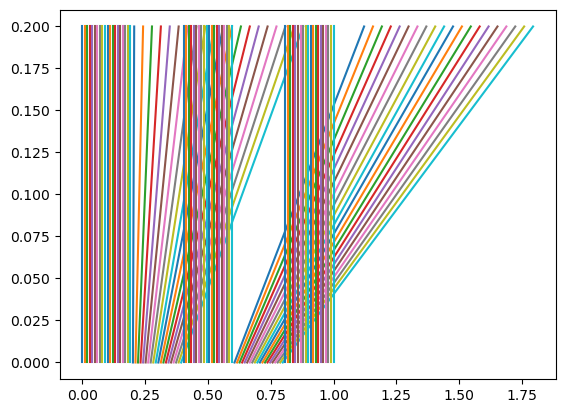}
    \end{minipage}

    \caption{Characteristics of the equation \eqref{eq:primitive}, where no rarefaction wave appears, and of \eqref{eq:primitiveCarillo}, where a rarefaction wave does appear. The initial condition is $u_0(x) := 0.4^{-1}\times (\indic_{0.2 < x < 0.4} + \indic_{0.6<x<0.8})$, and $m=2$. Time is the $y$ axes, and space is the $x$ axes. }
\end{figure}

\begin{prop}[Existence of a two-vortex viscosity solution]\label{prop:twovortex}
Let $m\ge 1$ and $\alpha \in (0,1)$. Consider the Equation~\eqref{eq:primitive}
with boundary conditions
\begin{equation}\label{eq:boundary_two}
  k(t,0)=0,\quad k(t,1)=\bar u_0 \quad\text{for } t\in(0,T),
\end{equation}
and initial datum $k(0,\cdot)=k_0(\cdot)$. Assume there exist $S_1^0,S_2^0,S_3^0,S_4^0$ such that
\begin{equation}\label{eq:ordering-two}
  0\le S_1^0<S_2^0\le \alpha \le S_3^0<S_4^0\le 1,
\end{equation}
\begin{equation*}
    k_0(s) = \begin{cases}
        0 & s<S_1^0, \\
       \displaystyle \al \bar u_0 \frac{s-S_1^0}{S_2^0-S_1^0} & S_1^0\le s<S_2^0, \\
        \al \bar u_0 & S_2^0\le s < S_3^0,\\
      \displaystyle  ( 1- \al) \bar u_0\frac{ s-S_3^0}{S_4^0-S_3^0} +\al\bar u_0 & S_3^0\le s < S_4^0, \\
        \bar u_0 & S_4^0\le s < 1.
    \end{cases}
\end{equation*}
Let $S_i\in C^1 ([0,T) )$, $i=1,2,3,4$, solve the Rankine-Hugoniot condition
\begin{align*}
  \frac{dS_1}{dt} &= -\alpha^{m-1}\bar u_0^{ m} \frac{S_1}{ (S_2-S_1 )^{m-1}}, 
  & \frac{dS_2}{dt} &= \alpha^{m-1}\bar u_0^{ m} \frac{\alpha-S_2}{ (S_2-S_1 )^{m-1}},\\
  \frac{dS_3}{dt} &= (1-\alpha)^{m-1}\bar u_0^{ m} \frac{\alpha-S_3}{ (S_4-S_3 )^{m-1}}, 
  & \frac{dS_4}{dt} &= (1-\alpha)^{m-1}\bar u_0^{ m} \frac{1-S_4}{ (S_4-S_3 )^{m-1}},
\end{align*}
with initial data $S_i(0)=S_i^0$, posed on the maximal time interval on which the ordering
\begin{equation}\label{eq:ordering-two-evol}
  0\le S_1(t)<S_2(t)\le \alpha \le S_3(t)<S_4(t)\le 1
\end{equation}
is preserved.
Define $k:(0,T)\times(0,1)\to\mathbb{R}$ by
\begin{equation}\label{def:twovortex}
    k(t,s) = \begin{cases}
        0 & s<S_1(t), \\
       \displaystyle \al \bar u_0 \frac{s-S_1(t)}{S_2(t)-S_1(t)} & S_1(t)\le s<S_2(t), \\
        \al \bar u_0 & S_2(t)\le s < S_3(t),\\
      \displaystyle  ( 1- \al) \bar u_0\frac{ s-S_3(t)}{S_4(t)-S_3(t)} +\al\bar u_0 & S_3(t)\le s < S_4(t), \\
        \bar u_0 & S_4(t)\le s < 1.
    \end{cases}
\end{equation}
Then $k\in C^0 ([0,T)\times[0,1] )$, satisfies the boundary conditions \eqref{eq:boundary_two} and $k(0,\cdot)=k_0$, and is a viscosity solution of Equation~\eqref{eq:primitive} on $(0,T)\times(0,1)$ in the sense of Definition~\ref{def:viscosity}.
\end{prop}

\begin{proof}
Set the slopes
$$
\beta_1(t):=\frac{\alpha \bar u_0}{S_2(t)-S_1(t)},\qquad
\beta_2(t):=\frac{(1-\alpha)\bar u_0}{S_4(t)-S_3(t)}.
$$
Note that $k$ is continuous and piecewise affine in $s$ with $\partial_sk\in\{0,\beta_1,\beta_2\}$. Moreover the boundary conditions are satisfied.

\underline{Step 1: Interior of the five parts}. On the constant regions $\{s<S_1\}$, $\{S_2<s<S_3\}$, and $\{s>S_4\}$ we have $k= 0,\alpha\bar u_0,\bar u_0$, so $D^\pm k=\{(0,0)\}$ and Equation~\eqref{eq:primitive} is satisfied. Using the Rankine-Hugoniot condition as in the single vortex case, the PDE is also satisfied in the interior of these parts.

\underline{Step 2: Semi-differentials at the interface}
Set $\Gamma_i:=\{(t,s): s=S_i(t)\}$. With similar computation to the single vortex case we obtain
\begin{align*}
&\Gamma_1:\quad D^{-}k(t,S_{1}(t))=\{(-S_1'p_2,p_2):\,0\le p_2\le \beta_1\},\quad D^{+}k(t,S_{1}(t))=\emptyset;\\
&\Gamma_2:\quad D^{+}k(t,S_{2}(t))=\{(-S_2'p_2,p_2):\,0\le p_2\le \beta_1\},\quad D^{-}k(t,S_{2}(t))=\emptyset;\\
&\Gamma_3:\quad D^{-}k(t,S_{3}(t))=\{(-S_3'p_2,p_2):\,0\le p_2\le \beta_2\},\quad D^{+}k(t,S_{3}(t))=\emptyset;\\
&\Gamma_4:\quad D^{+}k(t,S_{4}(t))=\{(-S_4'p_2,p_2):\,0\le p_2\le \beta_2\},\quad D^{-}k(t,S_{4 }(t))=\emptyset.
\end{align*}

\underline{Step 3: Viscosity solutions}
It only remains to prove that $k$ is a supersolution on $\Gamma_1, \Gamma_3$ and a subsolution on $\Gamma_2, \Gamma_4$. 

On $\Gamma_1$:
Here $k(t,S_1)=0$ and $(p_1,p_2)\in D^-k$ has $p_1=-S_1' p_2$, $0\le p_2\le\beta_1$. Then
$$
-S_1'p_2+p_2^{m}(0-S_1\bar u_0)
=p_2(-S_1'-\bar u_0 S_1\,p_2^{m-1})
\ge p_2\,\bar u_0 S_1(\beta_1^{m-1}-p_2^{m-1})\ge0,
$$
because $S_1'=-\beta_1^{m-1}\bar u_0 S_1$ and $0\le p_2\le\beta_1$. Thus $k$ is a supersolution on $\Gamma_1$. 

On $\Gamma_2$:
Here $k(t,S_2)=\alpha\bar u_0$, $(p_1,p_2)\in D^+k$ has $p_1=-S_2'p_2$, $0\le p_2\le\beta_1$. Thus
 $$
-S_2'p_2+p_2^{m}\bar u_0(\alpha-S_2)
=p_2\,\bar u_0(\alpha-S_2) (p_2^{m-1}-\beta_1^{m-1} )\le0,
 $$
since $S_2'=\beta_1^{m-1}\bar u_0(\alpha-S_2)$ and by assumption $S_2\le\alpha$, hence $\alpha-S_2\ge0$ and $p_2^{m-1}\le\beta_1^{m-1}$. Thus $k$ is a subsolution on $\Gamma_2$. 

On $\Gamma_3$:
Here $k(t,S_3)=\alpha\bar u_0$, $(p_1,p_2)\in D^-k$ has $p_1=-S_3'p_2$, $0\le p_2\le\beta_2$. Then
 $$
-S_3'p_2+p_2^{m}\bar u_0(\alpha-S_3)
=p_2\,\bar u_0(\alpha-S_3) (p_2^{m-1}-\beta_2^{m-1} )\ge0,
 $$
because $S_3'=\beta_2^{m-1}\bar u_0(\alpha-S_3)$ and $S_3\ge\alpha$ implies $\alpha-S_3\le0$. Thus $k$ is a supersolution on $\Gamma_3$. 

On $\Gamma_4$:
Here $k(t,S_4)=\bar u_0$, $(p_1,p_2)\in D^+k$ has $p_1=-S_4'p_2$, $0\le p_2\le\beta_2$. Hence
 $$
-S_4'p_2+p_2^{m}\bar u_0(1-S_4)
=p_2\,\bar u_0(1-S_4) (p_2^{m-1}-\beta_2^{m-1} )\le0,
 $$
since $S_4'=\beta_2^{m-1}\bar u_0(1-S_4)$ and $1-S_4\ge0$. Thus $k$ is a subsolution on $\Gamma_4$. 

\end{proof}

\printbibliography

@article {MR3471104,
    AUTHOR = {Craig, Katy and Bertozzi, Andrea L.},
     TITLE = {A blob method for the aggregation equation},
   JOURNAL = {Math. Comp.},
  FJOURNAL = {Mathematics of Computation},
    VOLUME = {85},
      YEAR = {2016},
    NUMBER = {300},
     PAGES = {1681--1717},
      ISSN = {0025-5718,1088-6842},
   MRCLASS = {65M75 (35Q35 35Q82 65M15 82C22)},
  MRNUMBER = {3471104},
MRREVIEWER = {B\"ulent\ Karas\"ozen},
       DOI = {10.1090/mcom3033},
       URL = {https://doi.org/10.1090/mcom3033},
}

@article {MR3913840,
    AUTHOR = {Carrillo, Jos\'e{} Antonio and Craig, Katy and Patacchini,
              Francesco S.},
     TITLE = {A blob method for diffusion},
   JOURNAL = {Calc. Var. Partial Differential Equations},
  FJOURNAL = {Calculus of Variations and Partial Differential Equations},
    VOLUME = {58},
      YEAR = {2019},
    NUMBER = {2},
     PAGES = {Paper No. 53, 53},
      ISSN = {0944-2669,1432-0835},
   MRCLASS = {35Q84 (35Q35 35Q82 65M12 65M75 82C22)},
  MRNUMBER = {3913840},
MRREVIEWER = {Yifu\ Wang},
       DOI = {10.1007/s00526-019-1486-3},
       URL = {https://doi.org/10.1007/s00526-019-1486-3},
}

@article{DaoDiaz2025,
  author       = {Anh Dao Nguyen and Jesús Ildefonso Díaz},
  title        = {Uniqueness and qualitative behavior of solutions to a porous medium equation with Newtonian potential pressure of nonlinear density function},
  journal      = {Discrete and Continuous Dynamical Systems},
  volume       = {45},
  number       = {11},
  year         = {2025},
  pages        = {4454--4489},
  doi          = {10.3934/dcds.2025063},
}

@article {MR267257,
    AUTHOR = {Kru\v zkov, S. N.},
     TITLE = {First order quasilinear equations with several independent
              variables},
   JOURNAL = {Mat. Sb. (N.S.)},
  FJOURNAL = {Matematicheski\u i\ Sbornik. Novaya Seriya},
    VOLUME = {81(123)},
      YEAR = {1970},
     PAGES = {228--255},
      ISSN = {0368-8666},
   MRCLASS = {35.37},
  MRNUMBER = {267257},
MRREVIEWER = {Z.\ Zielezny},
}

@article {MR146673,
    AUTHOR = {O'Neil, Richard},
     TITLE = {Convolution operators and {$L(p,\,q)$} spaces},
   JOURNAL = {Duke Math. J.},
  FJOURNAL = {Duke Mathematical Journal},
    VOLUME = {30},
      YEAR = {1963},
     PAGES = {129--142},
      ISSN = {0012-7094,1547-7398},
   MRCLASS = {47.25 (46.35)},
  MRNUMBER = {146673},
MRREVIEWER = {I.\ I.\ Hirschman, Jr.},
       URL = {http://projecteuclid.org/euclid.dmj/1077374532},
}

@article{perthame_existence_2006,
	AUTHOR = {Perthame, Beno\^{i}t and Dalibard, Anne-Laure},
     TITLE = {Existence of solutions of the hyperbolic {K}eller-{S}egel
              model},
   JOURNAL = {Trans. Amer. Math. Soc.},
  FJOURNAL = {Transactions of the American Mathematical Society},
    VOLUME = {361},
      YEAR = {2009},
    NUMBER = {5},
     PAGES = {2319--2335},
      ISSN = {0002-9947},
   MRCLASS = {35L60 (35D30 92C17)},
  MRNUMBER = {2471920},
MRREVIEWER = {Lucilla Corrias},
       DOI = {10.1090/S0002-9947-08-04656-4},
       URL = {https://doi-org.accesdistant.sorbonne-universite.fr/10.1090/S0002-9947-08-04656-4}
}

@article {MR3124759,
    AUTHOR = {Tao, Youshan},
     TITLE = {Global dynamics in a higher-dimensional repulsion chemotaxis
              model with nonlinear sensitivity},
   JOURNAL = {Discrete Contin. Dyn. Syst. Ser. B},
  FJOURNAL = {Discrete and Continuous Dynamical Systems. Series B. A Journal
              Bridging Mathematics and Sciences},
    VOLUME = {18},
      YEAR = {2013},
    NUMBER = {10},
     PAGES = {2705--2722},
      ISSN = {1531-3492,1553-524X},
   MRCLASS = {35K51 (35A09 35B35 35B45 35K57 35Q92 92C17)},
  MRNUMBER = {3124759},
MRREVIEWER = {Lubomira\ G.\ Softova},
       DOI = {10.3934/dcdsb.2013.18.2705},
       URL = {https://doi.org/10.3934/dcdsb.2013.18.2705},
}

@article {MR3723527,
    AUTHOR = {Lai, Yulin and Xiao, Youjun},
     TITLE = {Existence and asymptotic behavior of global solutions to
              chemorepulsion systems with nonlinear sensitivity},
   JOURNAL = {Electron. J. Differential Equations},
  FJOURNAL = {Electronic Journal of Differential Equations},
      YEAR = {2017},
     PAGES = {Paper No. 254, 9},
      ISSN = {1072-6691},
   MRCLASS = {35M33 (35B40 35Q92)},
  MRNUMBER = {3723527},
}

@article {MR1739596,
    AUTHOR = {Lin, Fanghua and Zhang, Ping},
     TITLE = {On the hydrodynamic limit of {G}inzburg-{L}andau vortices},
   JOURNAL = {Discrete Contin. Dynam. Systems},
  FJOURNAL = {Discrete and Continuous Dynamical Systems},
    VOLUME = {6},
      YEAR = {2000},
    NUMBER = {1},
     PAGES = {121--142},
      ISSN = {1078-0947,1553-5231},
   MRCLASS = {35Q55 (82D50)},
  MRNUMBER = {1739596},
MRREVIEWER = {Robert\ Leon\ Jerrard},
       DOI = {10.3934/dcds.2000.6.121},
       URL = {https://doi.org/10.3934/dcds.2000.6.121},
}

@article {MR2448650,
    AUTHOR = {Dolbeault, Jean and Nazaret, Bruno and Savar\'e, Giuseppe},
     TITLE = {A new class of transport distances between measures},
   JOURNAL = {Calc. Var. Partial Differential Equations},
  FJOURNAL = {Calculus of Variations and Partial Differential Equations},
    VOLUME = {34},
      YEAR = {2009},
    NUMBER = {2},
     PAGES = {193--231},
      ISSN = {0944-2669,1432-0835},
   MRCLASS = {49J40 (49Q20)},
  MRNUMBER = {2448650},
MRREVIEWER = {Luca\ Granieri},
       DOI = {10.1007/s00526-008-0182-5},
       URL = {https://doi.org/10.1007/s00526-008-0182-5},
}

@Article{keller1971model,
  title={Model for chemotaxis},
  author={Keller, Evelyn F and Segel, Lee A},
  journal={J. Theor. Biol.},
  volume={30},
  number={2},
  pages={225--234},
  year={1971},
  publisher={Elsevier}
}

@article{keller_initiation_1970,
AUTHOR = {Keller, Evelyn F. and Segel, Lee A.},
     TITLE = {Initiation of slime mold aggregation viewed as an instability},
   JOURNAL = {J. Theor. Biol.},
  FJOURNAL = {Journal of Theoretical Biology},
    VOLUME = {26},
      YEAR = {1970},
    NUMBER = {3},
     PAGES = {399--415},
      ISSN = {0022-5193},
   MRCLASS = {92C17 (35Q92 82C24 92C45)},
  MRNUMBER = {3925816},
       DOI = {10.1016/0022-5193(70)90092-5},
       URL = {https://doi.org/10.1016/0022-5193(70)90092-5},
}

@article{ACDC25,
    author  =   {Chodron de Courcel, Antonin},
    title   =   {On clogged and fast diffusions in porous media with fractional pressure},
    year    =   {2025},
    journal =   {arXiv preprint arXiv:2501.11645},
}

@article{CGV22FastRegularisation,
    author      = {José A. Carrillo and David Gómez-Castro and Juan Luis Vázquez},
    title       = {A fast regularisation of a Newtonian vortex equation.},
    journal     = {Ann. Inst. H. Poincaré Anal. Non Linéaire},
    year        = {2022},
    number      = {3},
    pages       = {705-747}
}

@article{CGV2022vortex,
  title     =   {Vortex formation for a non-local interaction model with Newtonian repulsion and superlinear mobility},
  author    =   {Carrillo, Jose A and G{\'o}mez-Castro, David and V{\'a}zquez, Juan Luis},
  journal   =   {Advances in Nonlinear Analysis},
  volume    =   {11},
  number    =   {1},
  pages     =   {937--967},
  year      =   {2022},
  publisher =   {De Gruyter}
}

@article{belgacemjabin2013,
  title     =   {Compactness for nonlinear continuity equations},
  author    =   {Belgacem, Fethi Ben and Jabin, Pierre-Emmanuel},
  journal   =   {Journal of Functional Analysis},
  volume    =   {264},
  number    =   {1},
  pages     =   {139--168},
  year      =   {2013},
  publisher =   {Elsevier}
}

@book{Rakotoson2008,
  author    = {Rakotoson, Jean-Michel},
  title     = {Réarrangement Relatif: Un instrument d’estimations dans les problèmes aux limites},
  series    = {Mathématiques et Applications},
  publisher = {Springer},
  year      = {2008},
  pages     = {292},
  isbn      = {978-3-540-69118-1},
  language  = {fr}
}

@article{chapman1996mean,
  title={A mean-field model of superconducting vortices},
  author={Chapman, Stephen J and Rubinstein, Jacob and Schatzman, Michelle},
  journal={European Journal of Applied Mathematics},
  volume={7},
  number={2},
  pages={97--111},
  year={1996},
  publisher={Cambridge University Press}
}

@article{weinan1994dynamics,
  title={Dynamics of vortices in Ginzburg-Landau theories with applications to superconductivity},
  author={Weinan, E},
  journal={Physica D: Nonlinear Phenomena},
  volume={77},
  number={4},
  pages={383--404},
  year={1994},
  publisher={Elsevier}
}

@article{CarrilloLisiniSavareSlepcev2010,
  author  = {J. A. Carrillo and S. Lisini and G. Savaré and D. Slepčev},
  title   = {Nonlinear mobility continuity equations and generalized displacement convexity},
  journal = {Journal of Functional Analysis},
  year    = {2010},
  volume  = {258},
  number  = {5},
  pages   = {1273--1309},
  doi     = {10.1016/j.jfa.2009.10.016},
  url     = {https://doi.org/10.1016/j.jfa.2009.10.016}
}

@article{STAN2016infinitespeed,
    title   = {Finite and infinite speed of propagation for porous medium equations with nonlocal pressure},
    journal = {Journal of Differential Equations},
    volume  = {260},
    number  = {2},
    pages   = {1154-1199},
    year    = {2016},
    issn    = {0022-0396},
    doi     = {https://doi.org/10.1016/j.jde.2015.09.023},
    url     = {https://www.sciencedirect.com/science/article/pii/S0022039615004787},
    author  = {Diana Stan and Félix {del Teso} and Juan Luis Vázquez}
}

@article{stanTesoVazquez2019existence,
    title   =   {Existence of weak solutions for a general porous medium equation with nonlocal pressure},
    author  =   {Stan, Diana and Del Teso, Felix and V{\'a}zquez, Juan Luis},
    journal =   {Archive for Rational Mechanics and Analysis},
    volume  =   {233},
    pages   =   {451--496},
    year    =   {2019},
    publisher=  {Springer}
}

@article{nguyen2018porous,
  title     =   {Porous medium equation with nonlocal pressure in a bounded domain},
  author    =   {Nguyen, Quoc-Hung and V{\'a}zquez, Juan Luis},
  journal   =   {Communications in Partial Differential Equations},
  volume    =   {43},
  number    =   {10},
  pages     =   {1502--1539},
  year      =   {2018},
  publisher =   {Taylor \& Francis}
}

@article{serfaty2020mean,
  title={Mean field limit for Coulomb-type flows},
  author={Sylvia Serfaty and appendix with Mitia Duerinckx},
  journal={Duke Mathematical Journal},
  year={2018},
  url={https://api.semanticscholar.org/CorpusID:90261082}
}

\end{document}